\documentclass[11pt]{article}
\newtheorem{theorem}{Theorem}

\newtheorem{proof}{Proof}
\usepackage{hyperref}
\newtheorem{remark}{Remark}
\usepackage{amsmath}
\usepackage{graphicx}
\usepackage{amssymb}
\usepackage{graphicx}
\usepackage[dvipsnames,usenames]{color}
\usepackage{multirow}
\usepackage{array}
\title{Stationary analysis of certain Markov-modulated reflected random walks in the quarter plane}

\author{Ioannis Dimitriou\\ 
Department of Mathematics, 
University of Patras, P.O.~Box 
26500, Patras, Greece\\
E-mail: idimit@math.upatras.gr\\Website: \href{https://thalis.math.upatras.gr/~idimit/}{https://thalis.math.upatras.gr/~idimit/}}
\begin{document}
\maketitle
\begin{abstract}
In this work, we focus on the stationary analysis of a specific class of continuous time Markov-modulated reflected random walks in the quarter plane with applications in the modelling of two-node Markov-modulated queueing networks with coupled queues. The transition rates of the two-dimensional process depend on the state of a finite state Markovian background process. Such a modulation is space homogeneous in the set of inner states of the two-dimensional lattice but may be different in the set of states at its boundaries. To obtain the stationary distribution, we apply the power series approximation method, and the theory of Riemann boundary value problems. We also obtain explicit expressions for the first moments of the stationary distribution under some symmetry assumptions. An application in the modelling of a priority retrial system with coupled orbit queues is also presented. Using a queueing network example, we numerically validated the theoretical findings.\vspace{2mm}
\end{abstract}

\section{Introduction}
Our primary aim in this work is in methodology for obtaining stationary metrics of a certain Markov-modulated two-dimensional reflected random walks, as a general model describing a two-queue network with a sort of coupling, operating in a random environment. In particular, we are dealing with a Markovian process $\{Z(t);t\geq0\}=\{(X_{1}(t),X_{2}(t),J(t));t\geq0\}$ where the two-dimensional process $\{(X_{1}(t),X_{2}(t));t\geq0\}$ defined on $\mathbb{Z}_{+}^{2}$ is called the level process. The transition rates of the two-dimensional process $\{(X_{1}(t),X_{2}(t));t\geq0\}$ depend on the state of the phase process $\{J(t);t\geq0\}$, which is defined on a finite set $\{0,1,\ldots,N\}$ as well as on the state of the level process. 

More precisely, the increments of the individual processes $\{X_{1}(t)\}$, $\{X_{2}(t)\}$ take values in $\{-1,0,1,\ldots\}$ when $J(t)=0$ and in $\{0,1,\ldots\}$ when $J(t)\in\{1,\ldots,N\}$. The modulation at the state of the process is space homogeneous, except possibly at the boundaries of $\mathbb{Z}_{+}^{2}$. Furthermore, the negative increment of the individual process $\{X_{1}(t)\}$ (resp. $\{X_{2}(t)\}$) is not only affected by the phase processes, but also on the state of $\{X_{2}(t)\}$ (resp. $\{X_{1}(t)\}$).  

A special example described by such a Markov-modulated reflected random walk is a two-node G-queueing network (see e.g. \cite{gel1,gel2,gel3} for an overview on G-networks) with simultaneous arrivals, coupled processors and service interruptions. We assume that three classes of jobs arrive according to independent Poisson processes. Class $P_{i}$, $i=1,2$ is routed to queue $i$, while a job of class $P_{0}$ is placed simultaneously at both queues. Queues with simultaneous arrivals have numerous applications in computer-communication and production systems in which jobs are split among a number of
different processors, communication channels or machines. Note also that the queues in such models are dependent due to the simultaneous arrivals. In general this makes
an exact analysis of the model very hard. 

Each job at queue $i$ requires exponentially distributed service time with rate $\nu_{i}$. The network is not fully reliable, but is subject to several types of service interruptions, i.e., several types of failed modes. Service interruptions may model breakdowns, malfunctions, as well as preventive maintenance etc. The network operates properly, i.e., the service stations provide service (operating mode), for an exponentially distributed time with rate $\theta_{0,j}$, and then, it switches to the failed mode $j$, $j=1,\ldots,N$, i.e., a type-$j$ interruption of the operating mode occurs. The network stays in failed mode $j$ for an exponentially distributed time, and then, switches either to operating mode 0 (with rate $\theta_{j,0}$), or to another failed mode $k$ (with rate $\theta_{j,k}$, $k\neq j$). 

When the network is in a failed mode of either type it cannot provide service. Moreover, the arrival rates of the jobs of either type depend on the type of the mode (either operating or failed). 

The network is equipped with coupled processors. More precisely, the service rate of a queue depends on the state of the other queue, i.e., queues are interacting with each other. Thus, when the network is in operating mode, and both queues are non-empty, queue 1 serves at a rate $w\nu_{1}$, and queue 2 at a rate $(1-w)\nu_{2}$. If only one queue is non-empty it serves at full capacity, i.e., with rate $\nu_{i}$. 

Signal generation at a queue depends also on the state of the other queue. More precisely if both queues are non-empty, signals arrive at queue $1$ (resp. 2) with rate $w\lambda_{1}^{-}$ (resp. $(1-w)\lambda_{2}^{-}$) and with probability $t_{12}$ (resp. $t_{21}$) triggers the instantaneous movement of a job from queue 1 (resp. queue 2) to queue 2 (resp. queue 1), or with probability $t_{10}$ (resp. $t_{20}$) cancels a job from queue 1 (resp. queue 2). Triggering signals serve to improve load balancing, while deleting signals serve as job cancellations in manufacturing or virus attacks in communication systems. Upon receiving service at queue 1 (resp. 2), the job is either routed to queue 2 (resp. 1) with probability $r_{12}$ (resp. $r_{21}$), or leaves the network with probability $1-r_{12}$ (resp. $1-r_{21}$).

Note here that for the model at hand the service rate of a queue depends on the state of the other queue, and that dependence becomes more apparent due to the presence of simultaneous arrivals. Furthermore, the impact of modulation due to service interruptions increases further the high level of dependence in queueing dynamics. 
\subsection{Related work}
Markov-modulated processes are processes that are driven by an underlying Markov process, in the sense that its transition parameters are affected by the state of the underlying process. The interested reader refer to \cite{pa,pra1,pra} for more information. 

Such type of stochastic processes arise in many queueing applications where systems evolve in random environment (i.e., the underlying process), that may model the irregularity of the arrival process (e.g., are rush-hour phenomena), the irregularity of the service process (e.g., servers' breakdowns, servers' vacations, availability of resources etc.) or both. Other applications are found in biology \cite{lef}, in reliability \cite{lef1} etc. The vast literature on this topic shed light on the stationary behaviour of such models. A variety of approaches have been used, but the most widely applied is the matrix-analytic method. We briefly mention \cite{gaver,kim,neuts,nun} (not exhaustive list). However, it is known that they demand computational effort due to the large number of matrix
computations. Alternatively, some authors tried to investigate special cases for which the stationary distributions adopt a simple product form, e.g., \cite{eco1},\cite{zhu1}. Plenty of work on multidimensional level process has been devoted in deriving asymptotic formulas of the stationary distributions and studying the stability conditions \cite{miya1,miya2,miya3,oz,oz1,oz3} (not exhaustive list). We also refer to \cite{fi1}, that dealt with a network of infinite server queues in which the joint probability generating function (pgf) is characterized in terms of a system of partial differential equations; see also \cite{ma3,ma2,ma1} for some recent studies on Markov-modulated queueing systems.

Quite recently, the authors in \cite{devin} studied the approximation of the probability of a large excursion in the
busy cycle of a constrained Markov-modulated random walk in the quarter plane. Moreover, the authors in \cite{walr} recently introduced the power series approximation (PSA) method to provide approximated stationary metrics in a \textit{non-modulated} random walk in the quarter plane describing a slotted time generalized processor sharing system of two queues. There, at the beginning of a slot, if both queues are nonempty, a type 1 (resp. type 2) customer is served with probability $\beta$ (resp. $1-\beta$); see also \cite{vanle2,vanle1}. A first step towards applying PSA method to modulated queues with coupled processors was recently given in \cite{dim,dimm}. Clearly, coupled processor models arise naturally when limited resources are dynamically shared among processors, e.g., \cite{borst0,dimpaptwc,dimpapad,fay1,fay,van}, as well as in assembly lines in manufacturing \cite{antr}.
\subsection{Contribution}
\paragraph{Fundamental contribution} By employing the generating function approach, we come up with a system of functional equations, and we perform two different methods to investigate the stationary behaviour of a class of Markov-modulated two-dimensional queueing models. It is evidently true that Markov-modulation complicates almost every aspect of the stochastic process that describe the queueing model at hand. 

With Markov-modulation, we increase the dimension of the process describing our model, and this aspect have several consequences: The fundamental functional equation, which is the first step of the analysis, has no longer a scalar form, but instead is given in a matrix form \eqref{matrix}. By exploiting the special structure of the process we reduce the problem of solving this matrix equation to the problem of solving a scalar functional equation. However, to accomplish this task we must uniquely solve the system of equations \eqref{e3}. Moreover, the derivation of the stability condition requires some computational effort that it is connected also with the previous task. Additional effort is required in order to prove some other technical issues as presented in Appendices. 

In Section \ref{appl}, we discussed a similar but slightly different setting with application in priority retrial systems with two coupled orbit queues. There, some further technical issues arise, but a similar approach is applied. Note that it is the first time that such a retrial system is analysed in the related literature.

Thus, the overall analysis becomes quite challenging, since the dynamics of the original two-dimensional queueing process (often called the level process) is affected by the changes of a secondary Markovian process, which models the environment in which the original process operates (often called the phase process). In such a case, the phase process affects the evolution of the level process in two ways: $i)$ The rates at which certain transitions in the level process occur depend on the state of the phase process, as well as on the state itself. Thus, a change in the phase might not immediately trigger a transition of the level process,
but changes its dynamics (indirect interaction). $ii)$ A phase change does trigger an immediate transition of the level process (direct interaction).
\begin{enumerate}
\item We extend the class of stochastic processes in which PSA method (see \cite{vanle1,vanle,vanle4,walr}) can be applied, to the case of Markov-modulated reflected random walks in the quarter plane with a sort of coupling (see also \cite{dim,dimm} for some initial work in this direction). We obtain power series expansions of the pgf of the joint stationary distribution for any state of the phase process. A recursive technique to derive their coefficients is also presented; see Section \ref{psa}.
\item We also show how the theory of Riemann boundary value problems \cite{ga,coh,bv} can be applied for such kind of processes; see Section \ref{sec:rbv}.
\item By considering the \textit{symmetry assumption} (see Section \ref{symm}) at phase $J(t)=0$, we obtain explicit expressions for the moments of the stationary distribution of the individual processes $\{X_{j}(t)\}$, $j=1,2$.
\end{enumerate}

The overall conclusion is that both PSA, and the theory of boundary value problems can be adapted to any Markov-modulated two-dimensional processes for which the \textit{matrix} functional equation \eqref{matrix} can reduced to a \textit{scalar} functional equation of the form given in \eqref{fe}.
\paragraph{Applications} With this methodological approach we are able to analytically investigate a variety of Markov-modulated two-node queueing networks with a sort of coupling among queues. Applications of such queueing models arise naturally in systems where limited resources are dynamically shared among processors and where random incidents, such as breakdowns, preventive maintenance, may interrupt the operating mode and affect its input parameters. 

A first example arises in cable TV networks which are upgraded
to enable bidirectional traffic among network terminations (NTs) and a centrally
located head end (HE) \cite{van,res}. Such a communication is accomplished by a two-stage protocol. At the first stage,
an NT  requests a certain number of slots from the HE to transmit its data. If the HE receives the request successfully initiates the second stage by sending a grant to the NT. However, both NT and HE shares the same channel, and thus, the capacity of the channel is divided between these two stages. A portion of channel capacity is dedicated to data transmission
of NTs already having a grant, and the rest is dedicated to requests of NTs not yet having a
grant. Clearly, hardware/software malfunctions may interrupt the normal operation. In our setting, this situation maybe modelled by a Markov-modulated tandem queue with coupled processors where service at station 1 represents the process of receiving the requests,
whereas service at station 2 represents the transmission of the data due to the
successfully received requests. The network operation is modulated by a two state Markov process. 

Potential applications maybe found in bandwidth sharing of data flows \cite{gui1} and in device-to-device communication \cite{vita} by taking into account the impact of fading or signal attenuation, in assembly lines in manufacturing \cite{antr} for which two operations on each job must be performed
using a limited service capacity, by incorporating the presence of machine breakdowns.

We can further slightly modify the approach to analyse  
two-class coupled retrial queues with limited priority line in front of the server; see Section \ref{appl}. Such a model has potential application in the performance modelling of relay-assisted wireless cooperative networks, which operate as follows: There is network of a finite number of source users, a finite number of relay nodes and a common destination node. The source users (i.e., streams of jobs) transmit packets to the destination node (i.e., the service station) with the cooperation of the relays (i.e., the orbits). If a transmission of a user's packet to the destination fails, the relays store it in their buffers and try to forward it to the destination after a random time period. However, due to the interdependence among queues at relays, their transmission rates are mutually affected (i.e., coupled orbits). The destination mode may also generates its own packets that stores them in a finite buffer (i.e., a priority line) and have to be transmitted outside the network. For more information regarding the performance of such systems see \cite{dimpeva,dimpapad,dimtuan}.
%
%

The rest of the paper is organised as follows. In Section \ref{sec:mod}, the Markov-modulated reflected random walk in quarter plane is described in detail, and a system of functional equations along with some preliminary results are presented in Section \ref{sec:fun}. The main result (see Theorem \ref{th0}) regarding the application of the power series approximation (PSA) method is presented in Section \ref{psa}. A discussion regarding the numerical derivation of moments of the stationary distribution is also given. In Section \ref{sec:rbv}, we provide a detailed analysis to obtain the pgfs of the stationary joint distribution in terms of a solution of a Riemann boundary value problem (see Theorem \ref{thw}). Explicit expressions for the moments of the stationary distribution by employing the \textit{symmetric assumption} at phase $J(t)=0$ is given in Section \ref{symm} (See Theorem \ref{th2}). In Section \ref{appl} we further present an application of the approach to priority retrial systems with coupled orbits. We mention here that this model is a modification of the general model described in Section \ref{sec:mod}, and reveals the flexibility of the approach we apply. Numerical results obtained by using the PSA, as well as some observations about how the system parameters affect the system performance for a near priority system are given in Section \ref{sec:num}. Numerical validation of the PSA using the explicit expressions of the symmetrical case derived in Section \ref{symm} is also given. The paper concludes in Section \ref{conc}. 
\section{Model description}\label{sec:mod}
Consider a two-dimensional Markovian process $\{(X_{1}(t),X_{2}(t))\}$ on $\mathbb{Z}_{2}^{+}$ and a background process $\{J(t)\}$ on a finite state space $S_{0}=\{0,1,\ldots,N\}$. Assume that each of $\{X_{1}(t)\}$, $\{X_{2}(t)\}$ is skip free from the right, which means that their increments take values in $\{-1,0,1,\ldots\}$ when $J(t)=0$ and in $\{0,1,\ldots\}$ when $J(t)\neq 0$. The joint process $\{Z(t),t\geq0\}=\{(X_{1}(t),X_{2}(t),J(t)),t\geq 0\}$ is Markovian with state space $E=\mathbb{Z}_{2}^{+}\times S_{0}$. Note that the simplest version of $\{Z(t)\}$ corresponds to a two-dimensional quasi birth-death (QBD) process \cite{oz1} for which i.e., the increments of $\{X_{j}(t)\}$, $j=1,2,$ are in $\{-1,0,1\}$ when $J(t)=0$, and in $\{0,1\}$ when $J(t)=1,\ldots,N$.


The infinitesimal generator matrix $Q$ of $\{Z(t),t\geq0\}$ is represented in block form as
\begin{displaymath}
Q=[Q_{(x_{1},x_{2}),(x_{1}^{\prime},x_{2}^{\prime})};(x_{1},x_{2}),(x_{1}^{\prime},x_{2}^{\prime})\in\mathbb{Z}_{+}^{2}],
\end{displaymath} 
where each block $Q_{(x_{1},x_{2}),(x_{1}^{\prime},x_{2}^{\prime})}$ is given by $Q_{(x_{1},x_{2}),(x_{1}^{\prime},x_{2}^{\prime})}=[q_{(x_{1},x_{2},j),(x_{1}^{\prime},x_{2}^{\prime},j^{\prime})}]$ for $(x_{1},x_{2},j),(x_{1}^{\prime},x_{2}^{\prime},j^{\prime})\in E$, and $q_{(x_{1},x_{2},j),(x_{1}^{\prime},x_{2}^{\prime},j^{\prime})}$, are the infinitesimal transition rates from state $(x_{1},x_{2},j)$ to $(x_{1}^{\prime},x_{2}^{\prime},j^{\prime})$. Let $\mathbb{H}=\{-1,0,1,\ldots\}$, $\mathbb{H}^{+}=\{0,1,\ldots\}$. The transition rates satisfy the following rules:
\begin{enumerate}
\item For $j=j^{\prime}=0$ (i.e., phase 0, see Figure \ref{fig:2} for the QBD version of $\{Z(t)\}$, i.e., when the increments are in $\{-1,0,1\}$.),
\begin{displaymath}
q_{(x_{1},x_{2},0),(x_{1}^{\prime},x_{2}^{\prime},0)}=\left\{\begin{array}{ll}
q_{\Delta x_{1},\Delta x_{2}}(0),&x_{1}\geq1,x_{2}\geq1,\Delta x_{1},\Delta x_{2}\in \mathbb{H},\\
q^{(1)}_{\Delta x_{1},\Delta x_{2}}(0),&x_{1}\geq1,x_{2}=0,\Delta x_{1}\in \mathbb{H},\Delta x_{2}\in \mathbb{H}^{+},\\
q^{(2)}_{\Delta x_{1},\Delta x_{2}}(0),&x_{1}=0,x_{2}\geq1,\Delta x_{1}\in \mathbb{H}^{+},\Delta x_{2}\in \mathbb{H},\\
q^{(0)}_{\Delta x_{1},\Delta x_{2}}(0),&x_{1}=x_{2}=0,\Delta x_{1},\Delta x_{2}\in \mathbb{H}^{+},
\end{array}\right.
\end{displaymath}
where $\Delta x_{k}=x_{k}^{\prime}-x_{k}$, $k=1,2$. Moreover, 
\begin{equation}
\begin{array}{rl}
q^{(k)}_{\Delta x_{1},\Delta x_{2}}(0)=&q_{\Delta x_{1},\Delta x_{2}}(0),\,k=1,2,\,\Delta x_{1},\Delta x_{2}\in \mathbb{H}^{+},\\
q_{-1,-1}(0)=&0,
\end{array}
\label{as1}
\end{equation}
and for $w\in[0,1]$,
\begin{displaymath}
\begin{array}{rl}
q_{i,-1}(0)=&(1-w)q^{(2)}_{i,-1}(0),\,i\in\mathbb{H}^{+}\\
q_{-1,j}(0)=&wq^{(1)}_{-1,j}(0),\,j\in\mathbb{H}^{+},
\end{array}
\end{displaymath}
or equivalently,
\begin{equation}
\begin{array}{rl}
\frac{q_{-1,j}(0)}{q_{-1,j}^{(1)}(0)}+\frac{q_{i,-1}(0)}{q_{i,-1}^{(2)}(0)}=&1,\,i,j\in\mathbb{H}^{+}.
\end{array}\label{coupl}
\end{equation}
\item For $j=j^{\prime}=1,\ldots,N$,
\begin{equation}
q_{(x_{1},x_{2},j),(x_{1}^{\prime},x_{2}^{\prime},j)}=q_{\Delta x_{1},\Delta x_{2}}(j),\,\Delta x_{1},\Delta x_{2}\in \mathbb{H}^{+},
\label{ass2}
\end{equation}
i.e., when $J(t)\neq 0$, no transitions are allowed to the West, North-West, South, South-West and South-East (see Figure \ref{fig:3} for the corresponding version of $\{Z(t)\}$ when the increments are only in $\{0,1\}$ for $J(t)=1,\ldots,N$.).
\item For $j\neq j^{\prime}$,
\begin{equation}
q_{(x_{1},x_{2},j),(x_{1}^{\prime},x_{2}^{\prime},j^{\prime})}=
\theta_{j,j^{\prime}}a_{\Delta x_{1},\Delta x_{2}}^{(j,j^{\prime})},\,\Delta x_{1},\Delta x_{2}\in \mathbb{H}^{+},
\label{as3}
\end{equation}
i.e., a phase change from $j$ to $j^{\prime}$ triggers a transition of $\{(X_{1}(t),X_{2}(t));t\geq0\}$ from $(x_{1},x_{2})$ to $(x^{\prime}_{1},x^{\prime}_{2})$ with probability $a_{\Delta x_{1},\Delta x_{2}}^{(j,j^{\prime})}$.  
\end{enumerate}

\begin{figure*}
  \includegraphics[width=0.75\textwidth]{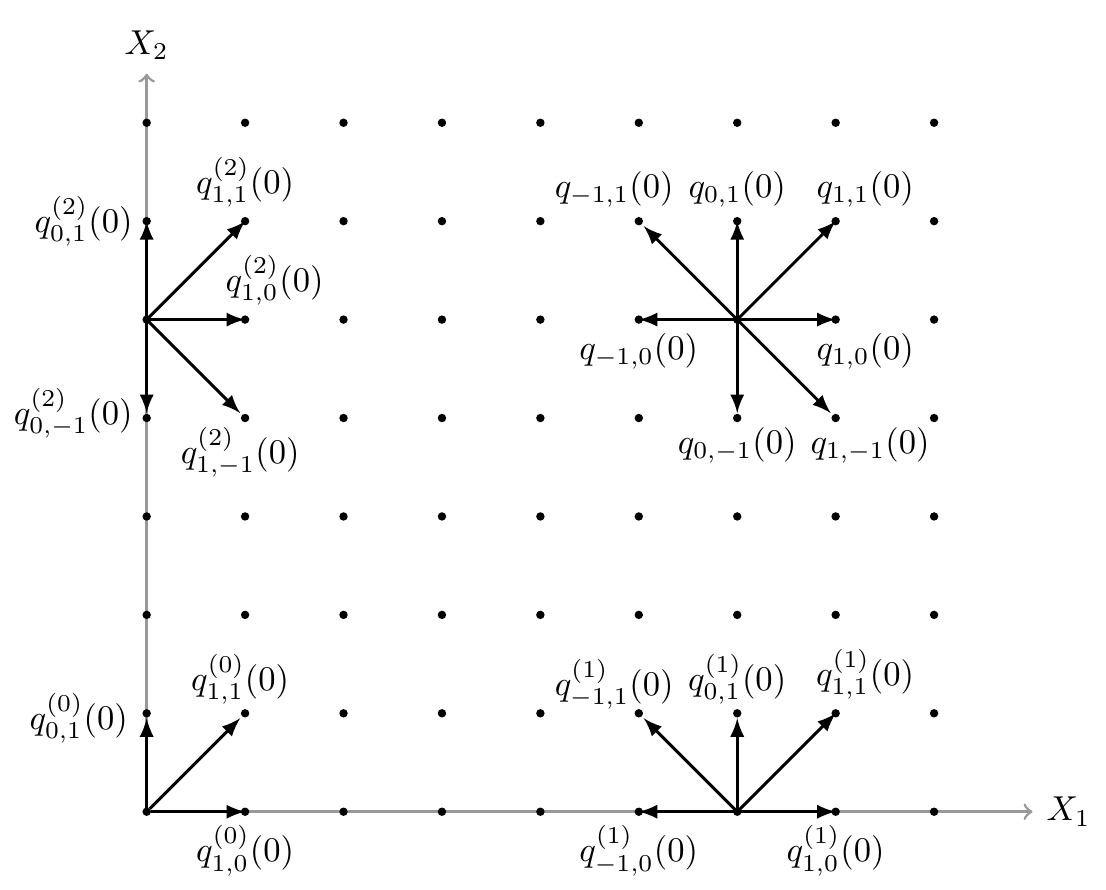}
\caption{State transition diagram of the QBD version of $\{Z(t),t\geq0\}$ at phase $J(t)=0$, i.e., the increments are in $\{-1,0,1\}$.}
\label{fig:2}       
\end{figure*}
\begin{figure*}
  \includegraphics[width=0.75\textwidth]{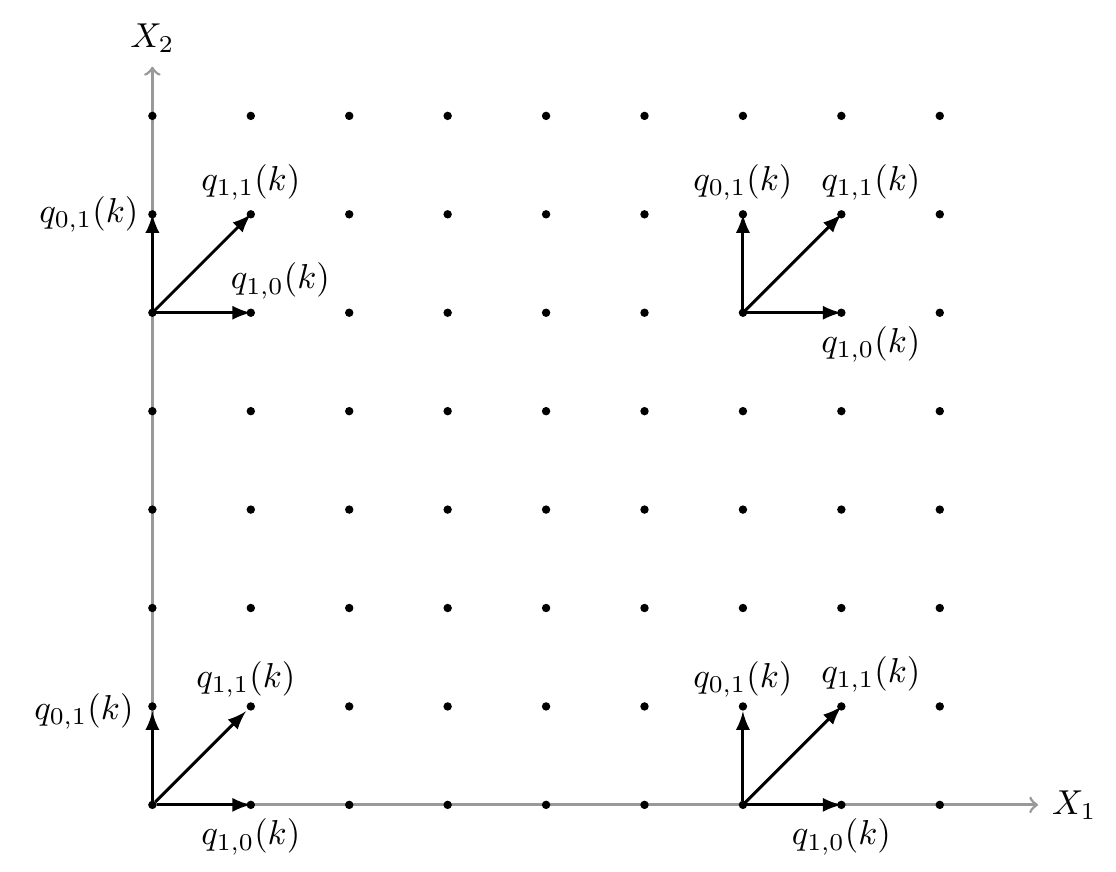}
\caption{State transition diagram of the QBD version of $\{Z(t),t\geq0\}$ at phase $J(t)=k$, $k=1,\ldots,N,$ i.e., the increments are in $\{0,1\}$.}
\label{fig:3}       
\end{figure*}
\begin{remark}
Note that the condition displayed in \eqref{as1} is \textbf{not} essential for the analysis that follows in Sections \ref{psa}, \ref{sec:rbv}. If it does not hold, the analysis can be easily modified. Condition \eqref{as1} is only useful in deriving easily the limiting probability of $\{Z(t);t\geq0\}$ at point $(0,0,0)$; see \eqref{lm}. 

Moreover, condition \eqref{coupl} \textit{is essential only} for the analysis in Section \ref{psa}. The analysis presented in Section \ref{sec:rbv} is general enough to consider also the case where condition \eqref{coupl} does not hold. However, some further technicalities are needed. Moreover, we can also easily modify the analysis in Section \ref{symm} to consider the case when \eqref{coupl} does not hold. Our main aim in this work is on the applicability of PSA \cite{walr} in the analysis of Markov-modulated reflected random walks in the quarter plane, and this is the reason why we employ condition \eqref{coupl}. 

Furthermore, conditions \eqref{ass2}, \eqref{as3} \textit{are essential} (in the sense that $\Delta x_{1},\Delta x_{2}$ should be in $\mathbb{H}^{+}$) both for the analysis in Section \ref{psa}, and for the one in Section \ref{sec:rbv}. This is because in case we allow $\Delta x_{1},\Delta x_{2}\in \mathbb{H}$, we introduce additional unknown \textit{boundary} functions that make the employed analysis intractable. 
\end{remark}

\section{The functional equations}\label{sec:fun}
Assume that the system is stable, and let the equilibrium probabilities
\begin{displaymath}
\pi_{i,j}^{(k)}=\lim_{t\to\infty}P((X_{1}(t),X_{2}(t),J(t))=(i,j,k)),\,(i,j,k)\in E.
\end{displaymath}
Let for $k=0,1,\ldots,N$
\begin{displaymath}
\Pi_{k}(x,y)=\sum_{i=0}^{\infty}\sum_{j=0}^{\infty}\pi_{i,j}^{(k)}x^{i}y^{j},|x|\leq 1,|y|\leq 1,
\end{displaymath}
and for $k,m=0,1,\ldots,N$, $k\neq m$,
\begin{displaymath}
A_{k,m}(x,y)=\sum_{i=0}^{\infty}\sum_{j=0}^{\infty}a_{i,j}^{(k,m)}x^{i}y^{j},|x|\leq 1,|y|\leq 1.
\end{displaymath}
By writing down the equilibrium equations we come up with the following system of functional equations
\begin{equation}
\begin{array}{rl}
R(x,y)\Pi_{0}(x,y)=&K(x,y)[(1-w)\Pi_{0}(x,0)-w\Pi_{0}(0,y)]\\&+C(x,y)\Pi_{0}(0,0)+xy\sum_{k=1}^{N}\theta_{k,0}A_{k,0}(x,y)\Pi_{k}(x,y),
\end{array}\label{a1}
\end{equation}
\begin{equation}
\Pi_{k}(x,y)=\frac{\sum_{m=0,m\neq k}^{N}\theta_{m,k}A_{m,k}(x,y)\Pi_{m}(x,y)}{D_{k}(x,y)},\,k=1,\ldots,N,
\label{e2}
\end{equation}
where for $k=0,1,\ldots,N$, 
\begin{displaymath}
\begin{array}{rl}
D_{k}(x,y)=&S_{k}(x,y)+\theta_{k,.},\\
S_{k}(x,y)=&\sum_{i=0}^{\infty}\sum_{j=0}^{\infty}q_{i,j}(k)(1-x^{i}y^{j})1_{\{(i,j)\neq(0,0)\}},\\
\theta_{k,.}=&\sum_{m\neq k}\theta_{k,m}.
\end{array}
\end{displaymath}
Equations \eqref{a1}, \eqref{e2} provide a relationship between the generating functions of the joint queue-length distribution for each network mode. Let $\mathbf{\Pi}(x,y)=(\Pi_{0}(x,y),\Pi_{1}(x,y),\ldots,\Pi_{N}(x,y))$, \eqref{a1}, \eqref{e2} and can be written in matrix form as
\begin{equation}
\begin{array}{l}
\mathbf{\Pi}(x,y)\mathbf{Q}(x,y)=[(1-w)\mathbf{\Pi}(x,0)-w\mathbf{\Pi}(0,y)]\mathbf{T}_{1}(x,y)+\mathbf{\Pi}(0,0)\mathbf{T}_{2}(x,y),	
\end{array}\label{matrix}
\end{equation}
where
\begin{displaymath}
\begin{array}{c}
\mathbf{Q}(x,y)=\begin{pmatrix}
R(x,y)&\theta_{0,1}A_{0,1}(x,y)&\ldots&\theta_{0,N}A_{0,N}(x,y)\\
-xy\theta_{1,0}A_{1,0}(x,y)&-D_{1}(x,y)&\ldots&\theta_{1,N}A_{1,N}(x,y)\\
\vdots&\vdots&\ddots&\vdots\\
-xy\theta_{N,0}A_{N,0}(x,y)&\theta_{N,1}A_{N,1}(x,y)&\ldots&-D_{N}(x,y)
\end{pmatrix},\vspace{2mm}\\
\mathbf{T}_{1}(x,y)=\begin{pmatrix}
K(x,y)&0&\ldots&0\\
0&0&\ldots&0\\
\vdots&\vdots&\ddots&\vdots\\
0&0&\ldots&0\\
\end{pmatrix},\,\,\mathbf{T}_{2}(x,y)=\begin{pmatrix}
C(x,y)&0&\ldots&0\\
0&0&\ldots&0\\
\vdots&\vdots&\ddots&\vdots\\
0&0&\ldots&0\\
\end{pmatrix}.
\end{array}
\end{displaymath}
Moreover,
\begin{displaymath}
\begin{array}{rl}
R(x,y)=&xy D_{0}(x,y)+y\sum_{j=0}^{\infty}q_{-1,j}(0)(x-y^{j})+x\sum_{i=0}^{\infty}q_{i,-1}(0)(y-x^{i}),\\
K(x,y)=&x\sum_{i=0}^{\infty}q^{(2)}_{i,-1}(0)(y-x^{i})-y\sum_{j=0}^{\infty}q^{(1)}_{-1,j}(0)(x-y^{j}),\\
C(x,y)=&wK(x,y)+y\sum_{j=0}^{\infty}q^{(1)}_{-1,j}(0)(x-y^{j})\\
=&-(1-w)K(x,y)+x\sum_{i=0}^{\infty}q^{(2)}_{i,-1}(0)(y-x^{i}).
\end{array}
\end{displaymath}

Equation \eqref{matrix} has the same form as equation (1.2) in the seminal paper \cite{fay1}. The major difference relies on the fact that in our
case we deal with a matrix fundamental equation. Due to its special structure, the problem of finding the solution of \eqref{matrix} can be reduced to the problem of finding the solution of a scalar fundamental form as follows. 

Note that equation \eqref{e2} defines a $N\times N$ system of linear equations
\begin{equation}
\mathbf{L}(x,y)^{T}\mathbf{P}(x,y)=\mathbf{E}(x,y)\Pi_{0}(x,y),\label{e3}
\end{equation}
where 
\begin{displaymath}
\begin{array}{rl}
\mathbf{P}(x,y)=&(\Pi_{1}(x,y),\ldots,\Pi_{N}(x,y))^{T},\\
\mathbf{E}(x,y)=&(\theta_{0,1}A_{0,1}(x,y),\ldots,\theta_{0,N}A_{0,N}(x,y))^{T},\\
\mathbf{L}(x,y)=&\begin{pmatrix}
D_{1}(x,y)&-\theta_{1,2}A_{1,2}(x,y)&-\theta_{1,3}A_{1,3}(x,y)&\ldots&-\theta_{1,N}A_{1,N}(x,y)\\
-\theta_{2,1}A_{2,1}(x,y)&D_{2}(x,y)&-\theta_{2,3}A_{2,3}(x,y)&\ldots&-\theta_{2,N}A_{2,N}(x,y)\\
\vdots&\vdots&\ddots&\vdots&\vdots\\
-\theta_{N,1}A_{N,1}(x,y)&-\theta_{N,2}A_{N,2}(x,y)&-\theta_{N,3}A_{N,3}(x,y)&\ldots&D_{N}(x,y)
\end{pmatrix}
\end{array}
\end{displaymath}
The system of equation \eqref{e3} yields
\begin{equation}
\begin{array}{rl}
\Pi_{k}(x,y)=&F_{0,k}(x,y)\Pi_{0}(x,y),\,k=1,\ldots,N,\\
F_{0,k}(x,y)=&\frac{1}{det[\mathbf{L}(x,y)^{T}]}(cof \mathbf{L}(x,y)^{T})^{T}\mathbf{E}(x,y),
\end{array}\label{a2}
\end{equation}
provided $det[\mathbf{L}(x,y)^{T}]\neq 0$, and where 
$cof \mathbf{L}(x,y)^{T}$ is the cofactor matrix of $\mathbf{L}(x,y)^{T}$. Note that $\mathbf{L}(x,y)$ is a non-singular matrix for $|x|,|y|=1$, since it is strictly diagonally dominant. Indeed,
\begin{displaymath}
\begin{array}{rl}
|D_{k}(x,y)|\geq&\sum_{i=0}^{\infty}\sum_{j=0}^{\infty}q_{i,j}(k)(1-|x|^{i}|y|^{j})1_{\{(i,j)\neq(0,0)\}}+\sum_{m=0,m\neq k}\theta_{k,m}\\
=&\sum_{m=0,m\neq k}\theta_{k,m}>\sum_{m\neq k,0}\theta_{k,m}=\sum_{m\neq k,0}\theta_{k,m}|A_{k,m}(x,y)|\\\geq&|\sum_{m\neq k,0}\theta_{k,m}A_{k,m}(x,y)|.
\end{array}
\end{displaymath}

Therefore, since $\sum_{k=0}^{N}\Pi_{k}(1,1)=1$, the stationary probabilities of the process $\{Z(t),t\geq0\}$ to be in any state of the background process $\{J(t), t\geq0\}$ are given by,
\begin{equation}
\begin{array}{rl}
\Pi_{0}(1,1)=&[1+\sum_{j=1}^{N}F_{0,j}(1,1)]^{-1},\\
\Pi_{k}(1,1)=&F_{0,k}(1,1)\Pi_{0}(1,1),\,k=1,\ldots,N.
\end{array}\label{prob}
\end{equation}

Using \eqref{a1}, \eqref{a2} we come up with the following functional equation:
\begin{equation}
\begin{array}{rl}
\Pi_{0}(x,y)[R(x,y)T(x,y)-xy]=&T(x,y)\{K(x,y)[(1-w)\Pi_{0}(x,0)\\&-w\Pi_{0}(0,y)]+C(x,y)\Pi_{0}(0,0)\},
\end{array}\label{fe}
\end{equation}
where $T(x,y)=[\sum_{k=1}^{N}\theta_{k,0}A_{k,0}(x,y)F_{0,k}(x,y)]^{-1}$.
\subsection{Stability condition}
To derive the stability condition, we look carefully at $K(x,y)$, and realize that we can find a function $x=s(y)$, such that 
$K(s(y),y)=0$, and more importantly that $s(y)\to1$, as $y\to 1$\footnote{Such an observation was also made in \cite{dimm,van}.}. This is the most important point in obtaining the stability condition. Simple numerical examples have shown that based on the values of the transition rates $q_{-1,j}^{(1)}$, $q_{i,-1}^{(2)}$, $i,j\in\mathbb{H}^{+}$, the zero $x:=s(y)$, for $|y|=1$, $y\neq 1$, may lie either inside, or outside, or on the unit circle; see Figure \ref{fig:w}.

Note that if we consider the QBD version of $\{Z(t)\}$, i.e., the increments of $\{X_{j}(t)\}$ are in $\{-1,0,1\}$ when $J(t)=0$, and in $\{0,1\}$ when $J(t)=1,\ldots,N$, we find explicitly this function:
\begin{displaymath}
\begin{array}{rl}
x:=s(y)=&\frac{-[y(q^{(1)}_{-1,0}(0)+q^{(1)}_{-1,1}(0)-q^{(2)}_{0,-1}(0)-q^{(2)}_{1,-1}(0))-q^{(2)}_{0,-1}(0)]+\sqrt{R(y)}}{2q^{(2)}_{1,-1}(0)},\\
R(y)=&[y(q^{(1)}_{-1,0}(0)+q^{(1)}_{-1,1}(0)-q^{(2)}_{0,-1}(0)-q^{(2)}_{1,-1}(0))-q^{(2)}_{0,-1}(0)]^{2}\\&+4q^{(2)}_{1,-1}(0)y(q^{(1)}_{-1,1}(0)y+q^{(1)}_{-1,0}(0)),
\end{array}
\end{displaymath}
and $\lim_{y\to 1}s(y)=1$. 

In the general case, it is rather impossible to find explicitly the function $s(y)$. In particular, $K(x,y)=0$ is equivalent with
\begin{equation}
\begin{array}{c}
x=W(x,y):=\frac{y(x\sum_{i=0}^{\infty}q^{(2)}_{i,-1}(0)+\sum_{j=0}^{\infty}q^{(1)}_{-1,j}(0)y^{j})}{\sum_{i=0}^{\infty}q^{(2)}_{i,-1}(0)x^{i}+y\sum_{j=0}^{\infty}q^{(1)}_{-1,j}(0)}.
\end{array}\label{bbb}
\end{equation}
Note that for $|x|\leq 1$, $|y|\leq 1$, the denominator in \eqref{bbb} never vanishes. Therefore, there exists a zero $x=s(y)$ such that $s(y)=W(s(y),y)$ with $s(1)=1$ (it is straightforward to see that for $y=1$, $x=1$ satisfies \eqref{bbb}). In any case, to obtain the stability condition it is neither essential to find an explicit expression for $x=s(y)$, nor to investigate its location for $|y|=1$, $y\neq 1$. We only need the value of its derivative with respect to $y$ at point $y=1$ (and of course to have $s(y)\to 1$ as $y\to 1$). Taking the derivative with respect to $y$ in $K(s(y),y)=0$, and letting $y\to1$ yields
\begin{displaymath}
\frac{ds(y)}{dy}|_{y=1}=\frac{\sum_{i=0}^{\infty}q^{(2)}_{i,-1}(0)+\sum_{j=0}^{\infty}jq^{(1)}_{-1,j}(0)}{\sum_{i=0}^{\infty}iq^{(2)}_{i,-1}(0)+\sum_{j=0}^{\infty}q^{(1)}_{-1,j}(0)}.
\end{displaymath}  

Thus, equation \eqref{fe} reduces to
\begin{equation}
\Pi_{0}(s(y),y)=\frac{T(s(y),y)C(s(y),y)}{R(s(y),y)T(s(y),y)-s(y)y}\Pi_{0}(0,0).
\label{stas}
\end{equation}
Letting $y\to 1$, and having in mind that $\lim_{y\to 1}s(y)=1$, \eqref{stas} yields (after applying once the L'Hospital rule) $\Pi_{0}(0,0)$ and the ergodicity condition of $\{Z(t),t\geq0\}$, i.e,
\begin{equation}
\begin{array}{rl}
\Pi_{0}(0,0)=&[1+\sum_{j=1}^{N}F_{0,j}(1,1)]^{-1}\\&\times\lim_{y\to 1}\left(\frac{\frac{d}{dy}[R(s(y),y)T(s(y),y)-s(y)y]}{\frac{d}{dy}[T(s(y),y)C(s(y),y)]}\right)>0.
\end{array}\label{lm}
\end{equation}
\begin{remark}
Ergodicity conditions for related two-dimensional QBD processes are also derived in \cite{oz,oz1} using the concept of induced Markov chains \cite{faymen}.
\end{remark}

\begin{figure*}
  \includegraphics[width=0.46\textwidth]{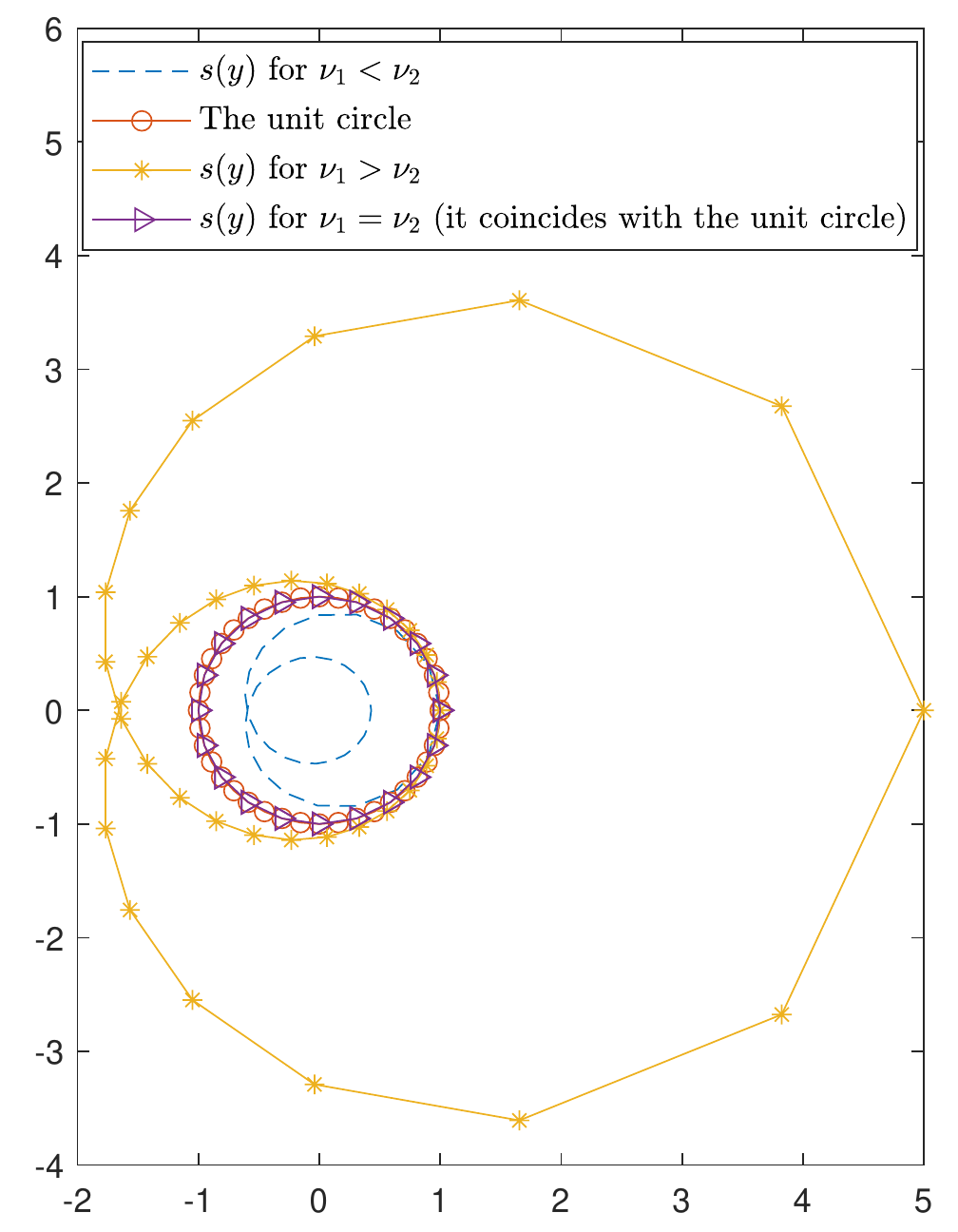}
\caption{The curve $x=s(y)$ for $q^{(1)}_{-1,1}(0)=\nu_{1}$, $q^{(1)}_{-1,j}(0)=0$, $j\neq 1$, $q^{(2)}_{0,-1}(0)=\nu_{2}$, $q^{(1)}_{i,-1}(0)=0$, $i\neq 0$.}
\label{fig:w}       
\end{figure*}
\subsection{A special case}\label{subse}
Consider the special case $\theta_{i,j}=0$, $i,j=1,\ldots,N$, i.e., the process switch always from any state $i\neq0$, only to state $0$. Then,
\begin{displaymath}
\mathbf{L}(x,y)=diag[D_{1}(x,y),\ldots,D_{N}(x,y)],\,\theta_{i,.}=\theta_{i,0},\,i=1,\ldots,N.
\end{displaymath}
In such a case,
\begin{equation}
\begin{array}{rl}
\Pi_{k}(x,y)=&\frac{\theta_{0,k}A_{0,k}(x,y)}{D_{k}(x,y)}\Pi_{0}(x,y),\,k=1,\ldots,N,\\
\Pi_{k}(1,1)=&\frac{\theta_{0,k}}{\theta_{k,0}}\Pi_{0}(1,1),
\end{array}\label{e4}
\end{equation}
and,
\begin{displaymath}
\Pi_{0}(1,1)=\frac{\prod_{j=1}^{N}\theta_{j,0}}{\sum_{k=0}^{N}\theta_{0,k}\prod_{j=0,j\neq k}^{N}\theta_{j,0}}.
\end{displaymath}
Thus, using the second in \eqref{e4} we can have the rest stationary probabilities the process $\{Z(t)\}$ to be in any phase $k$, $k=1,\ldots,N$.
Moreover, in such a case,
\begin{equation}
\begin{array}{rl}
T(x,y):=&\frac{D(x,y)}{\sum_{j=1}^{N}\theta_{j,0}A_{j,0}(x,y)\theta_{0,j}A_{0,j}(x,y)D^{(j)}(x,y)},\\
D^{(j)}(x,y):=&\prod_{k=1,k\neq j}^{N}D_{k}(x,y),\,j=1,\ldots,N,\\
D(x,y):=&\prod_{k=1}^{N}D_{k}(x,y).
\end{array} \label{edd}
\end{equation}
\subsection{The \textit{trivial} cases $w=0$ and $1$}
Note that when $w=0$ (resp. $w=1$), there are no transitions to West (resp. South) and to North-West (resp. South-East) from the interior states, when we are in phase $J(t)=0$, i.e., $q_{-1,0}(0)=q_{-1,1}(0)=0$ (resp. $q_{0,-1}(0)=q_{1,-1}(0)=0$). Such cases are not of our main interest but can serve as ``boundary" cases for the main subject of concern, which is the case $w\in(0,1)$ investigated in the following sections. In the following we only treat the case $w=0$. Analogous results can be obtained for the case $w=1$.

Note that for $w=0$, the functional equation \eqref{fe} reduces to
\begin{equation}
\begin{array}{rl}
\Pi_{0}(x,y)[R(x,y)T(x,y)-xy]=&T(x,y)[K(x,y)\Pi_{0}(x,0)+C(x,y)\Pi_{0}(0,0)].
\end{array}\label{fe1}
\end{equation}
It is not difficult to show by using Rouch\'e's theorem \cite{titc} (see Theorem \ref{th1} in Section \ref{sec:rbv}) that the left hand side of \eqref{fe1} has a single zero, say $y=\tilde{Y}(x)$, for $|x|=1$, such that $|\tilde{Y}(x)|<1$, and $\tilde{Y}(1)=1$. Then, after some algebra,
\begin{displaymath}
\Pi_{0}(x,y)=\frac{T(x,y)[K(x,y)C(x,\tilde{Y}(x))-K(x,\tilde{Y}(x))C(x,y)]}{K(x,\tilde{Y}(x))[xy-T(x,y)R(x,y)]}\Pi_{0}(0,0).
\end{displaymath}
The rest pgfs $\Pi_{k}(x,y)$, $k=1,\ldots,N$ can be obtained by using \eqref{a2}.
    
\section{The power series approximations in $w\in(0,1)$}\label{psa}
Starting by \eqref{fe}, our aim is to construct a power series expansion of the pgf
$\Pi_{0}(x, y)$ in $w$, and then, with this result to construct
power series expansions of $\Pi_{k}(x, y)$ in $w$ using \eqref{e4}. Let
\begin{displaymath}
\Pi_{k}(x,y)=\sum_{m=0}^{\infty}V_{m}^{(k)}(x,y)w^{m},\,k=0,1,\ldots,N.
\end{displaymath}

The major difficulty in solving \eqref{fe} corresponds to the presence of the two
unknown boundary functions $\Pi_{0}(x, 0)$, $\Pi_{0}(0, y)$. We proceed as in \cite{walr,dim,dimm} and observe that \eqref{fe} is rewritten as 
\begin{equation}
\begin{array}{l}
G(x,y)\Pi_{0}(x,y)-G_{10}(x,y)\Pi_{0}(x,0)-G_{00}(x,y)\Pi_{0}(0,0)\\
=wG_{10}(x,y)[\Pi_{0}(x,y)-\Pi_{0}(x,0)-\Pi_{0}(0,y)+\Pi_{0}(0,0)],
\end{array}\label{jk}
\end{equation}
where
\begin{equation}
\begin{array}{rl}
G(x,y)=&T(x,y)[yD_{0}(x,y)+\sum_{i=0}^{\infty}q_{i,-1}^{(2)}(0)(y-x^{i})]-y,\\
G_{10}(x,y)=&T(x,y)\{\sum_{i=0}^{\infty}q_{i,-1}^{(2)}(0)(y-x^{i})-y\sum_{j=0}^{\infty}q_{-1,j}^{(1)}(0)(1-x^{-1}y^{j})\},\\
G_{00}(x,y)=&yT(x,y)\sum_{j=0}^{\infty}q_{-1,j}^{(1)}(0)(1-x^{-1}y^{j}).
\end{array}
\end{equation}

The following theorem summarizes the main result.
\begin{theorem}\label{th0}
Under stability condition \eqref{lm},
\begin{equation}
\begin{array}{l}
V_{m}^{(0)}(x,y)=Q_{m-1}(x,y)\frac{G_{10}(x,y)}{G(x,y)},\,m>0,
\end{array}\label{res0}
\end{equation}
\begin{equation}
\begin{array}{rl}
V_{0}^{(0)}(x,y)=&\frac{G_{00}(x,y)G_{10}(x,Y(x))-G_{00}(x,Y(x))G_{10}(x,y)}{G(x,y)G_{10}(x,Y(x))}V_{0}^{(0)}(0,0),\vspace{2mm}\\
V_{m}^{(k)}(x,y)=&F_{0,k}(x,y)V_{m}^{(0)}(x,y),\,k=1,\ldots,N,\,m\geq0,
\end{array}
\label{res1}
\end{equation}
where $Y(x)$, $|x|\leq 1$ is the only zero of $G(x,y)$ in the unit disc $|y|\leq 1$, and for $m\geq 0$, 
\begin{displaymath}
Q_{m}(x,y)=V_{m}^{(0)}(x,y)-V_{m}^{(0)}(x,Y(x))-V_{m}^{(0)}(0,y)+V_{m}^{(0)}(0,Y(x)),
\end{displaymath}
with $Q_{-1}(x,y)=0$.
\end{theorem}\begin{proof}
Starting from \eqref{jk}, we observe that only $\Pi_{0}(0,y)$ appears with a factor $w$. By expanding the pgfs as power series
in $w$ and equating the coefficients of the corresponding powers of $w$, that unknown is eliminated. Indeed, we come up with
\begin{equation}
\begin{array}{rl}
G(x,y)V_{m}^{(0)}(x,y)=&G_{10}(x,y)[V_{m}^{(0)}(x,0)+P_{m-1}(x,y)]\\&+G_{00}(x,y)V_{m}^{(0)}(0,0),\,m\geq0,
\end{array}\label{fg}
\end{equation}
where
\begin{displaymath}
P_{m}(x,y)=V_{m}^{(0)}(x,y)-V_{m}^{(0)}(x,0)-V_{m}^{(0)}(0,y)+V_{m}^{(0)}(0,0),
\end{displaymath}
with $P_{-1}(x,y)=0$. Our aim is to express $V_{m}^{(0)}(x,y)$ in terms of $P_{-1}(x,y)$. Note now that there exist two unknowns, i.e., $V_{m}^{(0)}(x,0)$ and $V_{m}^{(0)}(0,0)$. By using Rouch\'e’s theorem \cite{adanrouche} we show in Appendix \ref{app1} that $G(x,y)=0$ has a unique root, $y=Y(x)$ such that $|Y(x)|<1$, $|x|=1$, $x\neq1$. Note that due to the \textit{implicit function theorem}, $Y(x)$
is an analytic function in the unit disc. Moreover, since $\Pi_{k}(x,y)$ is analytic for
all $x$ and $y$ in the unit disc, the terms $V_{m}^{(k)}(x,y)$ are also as well.

Thus, substituting in \eqref{fg} yields
\begin{equation}
V_{m}^{(0)}(x,0)=-\frac{G_{00}(x,Y(x))}{G_{10}(x,Y(x))}V_{m}^{(0)}(0,0)-P_{m-1}(x,Y(x)),
\label{za}
\end{equation}
and substituting \eqref{za} back in \eqref{fg} yields after some algebra the expressions for $V_{m}^{(0)}(x,y)$, $m\geq0$. Then, \eqref{e4} can be used to obtain the expressions for $V_{m}^{(k)}(x,y)$, $m\geq0$, $k=1,\ldots,N$ given in the second of \eqref{res1}.

It remains to calculate $V_{m}^{(0)}(0, 0)$ by using \eqref{lm}. Then, it follows that $V_{0}^{(0)}(0,0)=\Pi_{0}(0,0)$ and $V_{m}^{(0)}(0,0)=0$ for all $m > 0$. Hence starting from $V_{0}^{(0)}$ in \eqref{res0}, every $V_{m}^{(0)}$, $m>0$, and $V_{m}^{(k)}$, $m\geq0$, $k=1,\ldots,N$ can be determined iteratively via \eqref{res1}.
\end{proof}
\subsection{Performance metrics}
Having obtained iteratively the coefficients $V_{m}^{(k)}(x,y)$, $m\geq0$, $k=0,1,\ldots,N$, we are able to derive the moments of the stationary distribution of $\{Z(t);t\geq0\}$. The most interesting are 
\begin{equation*}
\begin{array}{rl}
E(X_{1})=&\sum_{m=0}^{\infty}w^{m}\frac{\partial}{\partial x}\sum_{k=0}^{N}V_{m}^{(k)}(x,1)|_{x=1},\vspace{2mm}\\
E(X_{2})=&\sum_{m=0}^{\infty}w^{m}\frac{\partial}{\partial y}\sum_{k=0}^{N}V_{m}^{(k)}(1,y)|_{y=1}.
\end{array}
\end{equation*}
Let $v_{m,1}=\frac{\partial}{\partial x}V_{m}^{(0)}(x,1)|_{x=1}.$ Then, using \eqref{a2}, \eqref{prob}, we finally obtain
\begin{equation}\begin{array}{l}
E(X_{1})=\frac{\sum_{k=1}^{N}\frac{\partial}{\partial x}F_{0,k}(x,1)|_{x=1}+[1+\sum_{k=1}^{N}F_{0,k}(1,1)]^{2}\sum_{m=0}^{\infty}w^{m}v_{m,1}}{1+\sum_{k=1}^{N}F_{0,k}(1,1)}.
\end{array}
\label{per}
\end{equation}
Similar expression can also be derived for $E(X_{2})$. By truncating the power series in (\ref{per}) we finally derive,
\begin{equation}
\begin{array}{rl}
E(X_{1})=&\frac{\sum_{k=1}^{N}\frac{\partial}{\partial x}F_{0,k}(x,1)|_{x=1}+[1+\sum_{k=1}^{N}F_{0,k}(1,1)]^{2}\sum_{m=0}^{M}w^{m}v_{m,1}+O(w^{M+1})}{1+\sum_{k=1}^{N}F_{0,k}(1,1)}.\end{array}
\label{pert}
\end{equation}
Clearly, the larger $M$ the better the approximation gets. Truncation yields accurate approximations for $w$ in the the neighborhood of 0. However, it is easy to note that the accurate calculation of the expressions in
(\ref{per}) requires the computation of the first derivatives of $V_{m}^{(0)}(x,y)$ for $m\geq 0$. It is evident that the calculation of these coefficients is far from straightforward due to the extensive use of L'Hospital's rule.

Obviously, our aim is to obtain an efficient approximation for the stationary metrics in the whole domain $w\in[0,1]$. Contrary to the truncated power series, we can use the Pad\'e approximants (as used in \cite{walr,vanle2}), which are rational functions of the form
\begin{displaymath}
[L/K](w)=\frac{\sum_{l=0}^{L}c_{1,l}w^{l}}{\sum_{k=0}^{K}c_{2,k}w^{k}},
\end{displaymath}
and have $L+K+2$ parameters to be set. As a normalization constant we can fix $c_{1,0} = 1$. The rest of them can be calculated from
the coefficients of the power series, by asking the derivatives of the Pad\'e approximant in $w=0$ and $1$ to be equal to the derivatives of the
power series of the performance metrics in $w = 0$ and $1$ (see \cite{walr,vanle2}).

\section{A Riemann boundary value problem}\label{sec:rbv}
Random walks in the quarter plane, which give rise to functional equations as given in \eqref{fe} are thoroughly discussed in \cite{coh,cohenRW,bv}. Our aim is to first obtain $\Pi_{0}(x,y)$ by formulating and solving a Riemann boundary value problem, and then, using that result to obtain the rest terms $\Pi_{k}(x,y)$ from \eqref{a2}.

The analysis of the kernel equation $U(x,y):=xy-\psi(x,y)$, where $\psi(x,y):=R(x,y)T(x,y)$, is the crucial step to obtain $\Pi_{0}(x,y)$, which is regular for $|x| < 1$, continuous for $|x|\leq 1$ for every fixed $y$ with $|y|\leq 1$; and similarly, with $x$, $y$ interchanged. Obviously, this means that $\Pi_{0}(x,y)$ should be finite for every zerotuple $(\widehat{x},\widehat{y})$ of the kernel $U(x,y)$ in $|x|\leq 1$, $|y|\leq 1$. Let 
\begin{displaymath}
A:=\{(x,y):U(x,y)=0,|x|\leq1,|y|\leq 1\}.
\end{displaymath} 
The approach is very briefly summarized as follows (see also \cite{coh,cohenRW,bv}):
\begin{enumerate}
\item  Restrict the analysis to a subset of $A$, in which
\begin{equation}
K(\widehat{x},\widehat{y})[(1-w)\Pi_{0}(\widehat{x},0)-w\Pi_{0}(0,\widehat{y})]+C(\widehat{x},\widehat{y})\Pi_{0}(0,0)=0,\,(\widehat{x},\widehat{y})\in A.
\label{fgg}
\end{equation}
\item Our aim is to construct through \eqref{fgg} the boundary function $\Pi_{0}(x,0)$ (resp. $\Pi(0,y)$) to be regular in $|x|<1$ (resp. $|y|<1$) and continuous in $|x|\leq 1$ (resp. $|y|\leq 1$). Then, $\Pi_{0}(x,y)$, and as a consequence, all the $\Pi_{k}(x,y)$, $k=1,\ldots,N$ can be obtained through \eqref{fe}. The key vehicle to accomplish this task is to find two curves $S_{1}\subset\{x\in\mathbb{C}:|x|\leq 1\}$, $S_{2}\subset\{y\in\mathbb{C}:|y|\leq 1\}$, and a  a one-to-one map $x = \omega(y)$ from $S_2$ to $S_1$ such that $(\omega(\widehat{y}), \widehat{y})$ is a zerotuple
of $U(x,y)$ for all $\widehat{y}\in S_{2}$. The analyticity of $U(x,y)$ in $x$, $y$ implies that by analytic continuation all zerotuples of the kernel can be constructed starting from $S_{2}$.
\end{enumerate}
\subsection{Kernel analysis \& preliminary results}
Note that $\psi(0,0)=0$, and thus, kernel analysis is done following the lines in \cite[Sections II.3.10-3.12]{bv}. Consider the kernel for 
\begin{equation}
x=gs,\,y=gs^{-1},\,|g|\leq 1,|s|=1.
\label{ker}
\end{equation}
\eqref{ker} defines for $s=e^{i\phi}$ a one-to-one mapping $f$ such that $x(\phi)=f(y(\phi))$ or $y(\phi)=f^{-1}(x(\phi))$, $\phi\in[0,2\pi)$. Moreover,
\begin{equation}
\begin{array}{rl}
U(gs,gs^{-1})=0\Leftrightarrow g^{2}=\psi(gs,gs^{-1}).
\end{array}
\end{equation}
\begin{theorem}\label{th1}
Under stability condition \eqref{lm}, the kernel $U(gs,gs^{-1})$ has in $|g|\leq 1$ exactly two zeros, of which one is identically zero. Denote the other zero by $g=g(s)$, where $g(1)=1$. Moreover, for $|s|=1$, $g(s)=-g(-s)$, $g(\bar{s})=\overline{g(s)}$. 
\end{theorem}
\begin{proof}
See Appendix \ref{app2}
\end{proof}

Let,
\begin{displaymath}
S_{1}:=\{x:x=g(s)s,|s|=1\},\,\,S_{2}:=\{y:y=g(s)s^{-1},|s|=1\},
\end{displaymath}
where $g(s)$ the positive zero of the kernel. 

To proceed we have to show that both $S_{1}$ and $S_{2}$ are \textit{simple} and \textit{smooth}. Although it is not difficult to show that under stability conditions, $S_{1}$, $S_{2}$ are both smooth (i.e., $\frac{d}{ds}g(s)$ exists in $|s|=1$), it is not possible to prove that, for general values of the system parameters, that the contours $S_{1}$ and $S_{2}$ are simply connected. This is due to the variety of possible locations of $x=0$, and $y=0$ with respect to $S_{1}$and $S_{2}$, respectively. Numerical results have shown that for small values of $w$ both $S_{1}$ and $S_{2}$ are simply connected; see Figures \ref{fig:4}, \ref{fig:5}. So, we restrict to the case where $S_{1}$ and $S_{2}$ are simple curves. In our case (see \cite[II.3.10]{bv}), this problem refers to the relation among $q_{-1,0}(0)$, $q_{0,-1}(0)$. In particular
\begin{enumerate}
\item If $q_{-1,0}(0)<q_{0,-1}(0)\Rightarrow x=0\in S_{1}^{+},\,y=0\in S_{2}^{-}$,
\item If $q_{-1,0}(0)=q_{0,-1}(0)\Rightarrow x=0\in S_{1},\,y=0\in S_{2}$,
\item If $q_{-1,0}(0)>q_{0,-1}(0)\Rightarrow x=0\in S_{1}^{-},\,y=0\in S_{2}^{+}$,
\end{enumerate}  
where $S_{j}^{+}$ (resp. $S_{j}^{-}$) denotes the interior (resp. the exterior) of $S_{j}$, $j=1,2$. Furthermore, if $s$ traverses the unit circle once, then $S_{1}$ is traversed twice anticlockwise, and $S_{2}$ is traversed twice clockwise.
\begin{figure*}
  \includegraphics[width=0.46\textwidth]{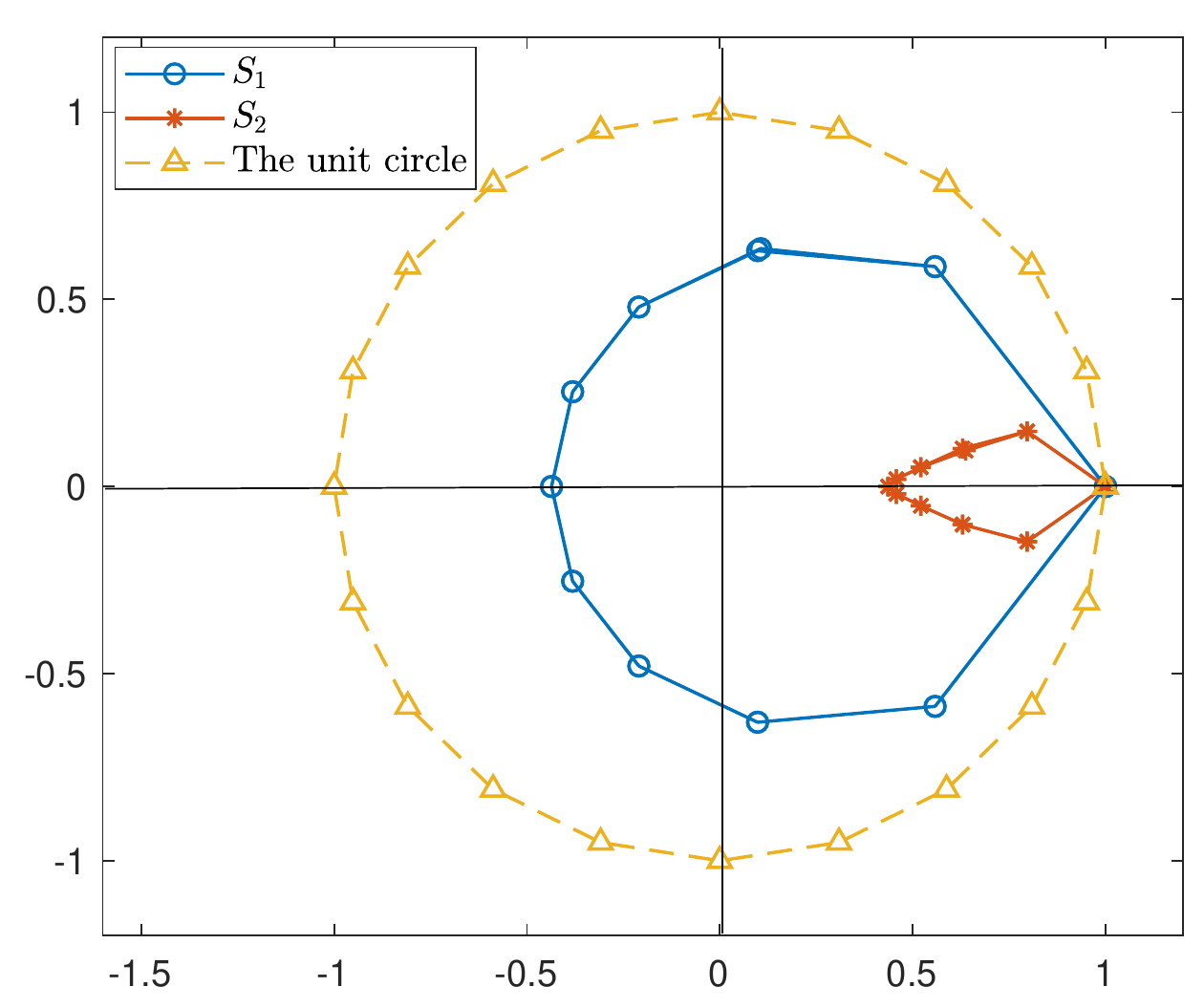}\,\,\includegraphics[width=0.45\textwidth]{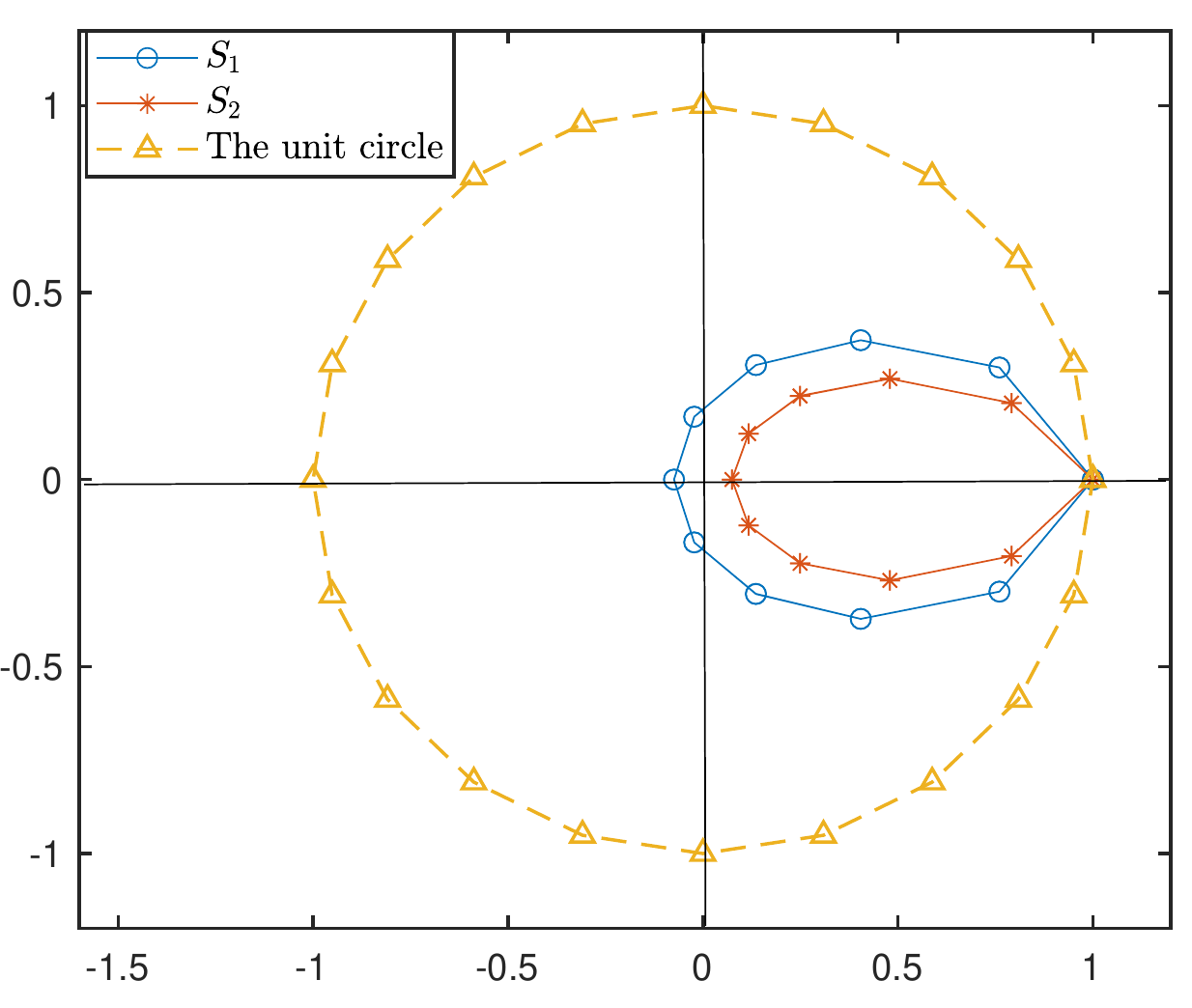}
\caption{The curves $S_{1}$, $S_{2}$ and the unit circle, when $w=0.1$ (left) and $w=0.5$ (right) ($\rho=0.4535<1$).}
\label{fig:4}       
\end{figure*}
\begin{figure*}
  \includegraphics[width=0.46\textwidth]{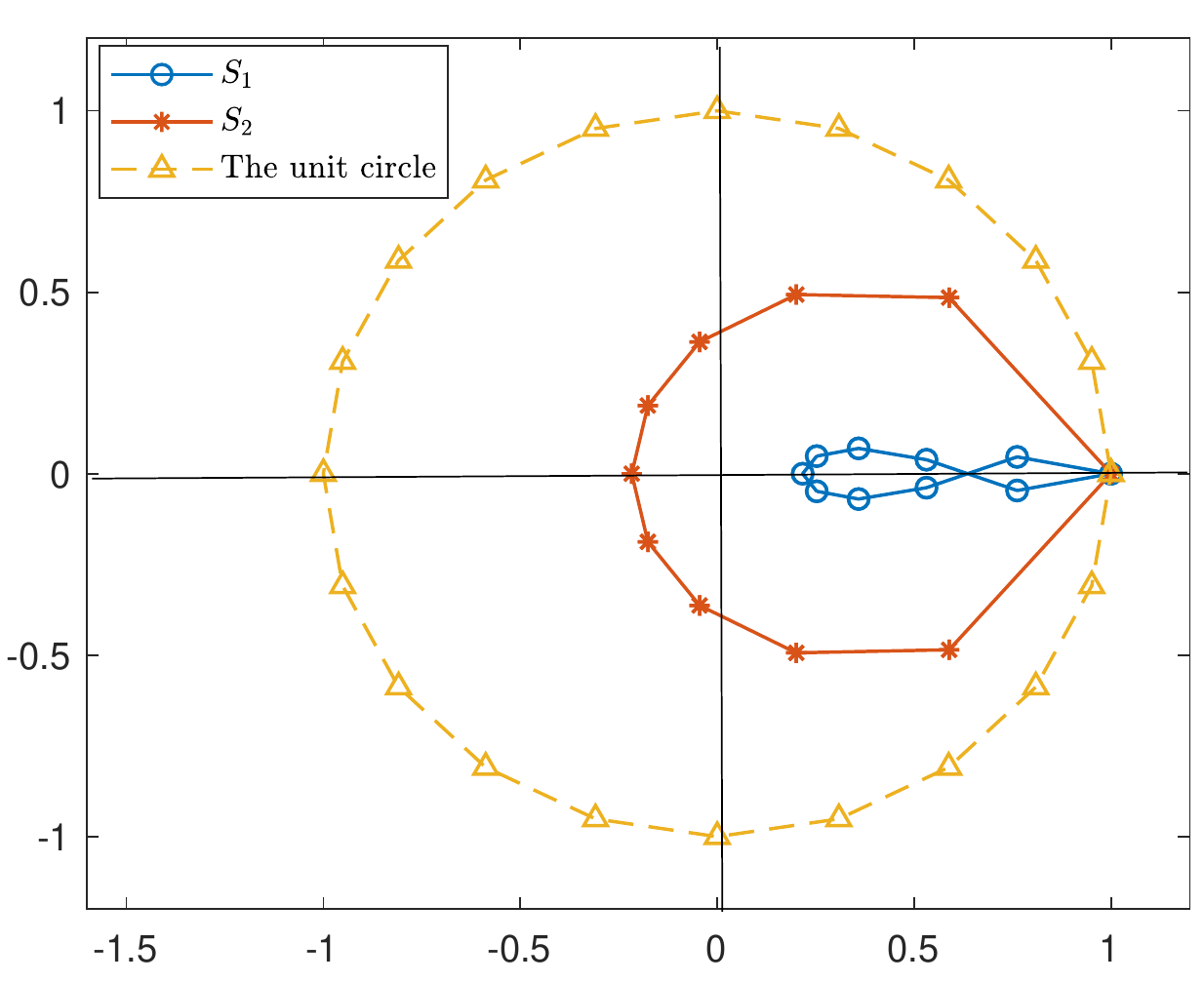}\,\,\includegraphics[width=0.46\textwidth]{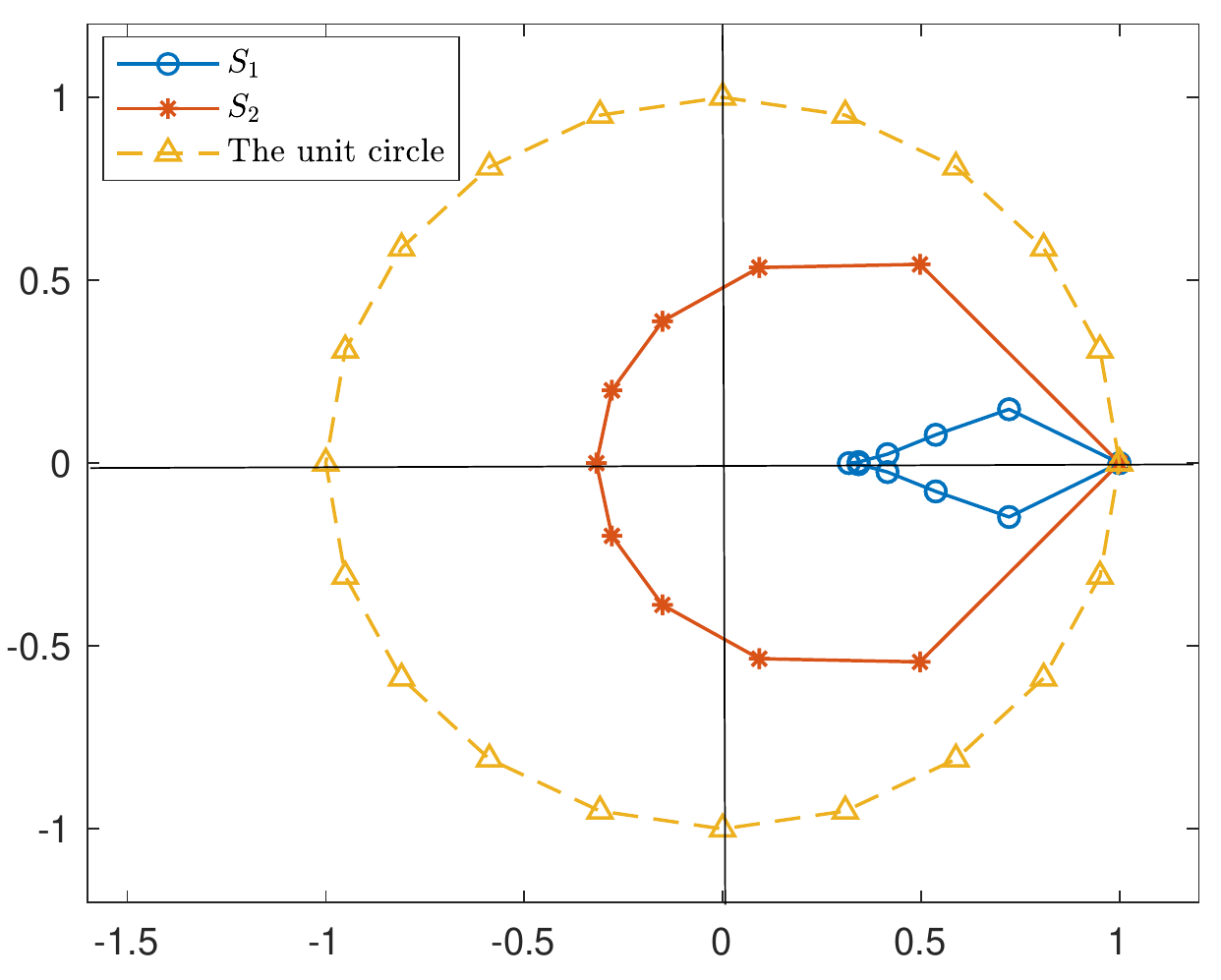}
\caption{The curves $S_{1}$, $S_{2}$ and the unit circle when $w=0.8$ (left) and $w=0.9$ (right) ($\rho=0.4535<1$).}
\label{fig:5}       
\end{figure*}


For the functions 
\begin{displaymath}
x=g(s)s,\,y=g(s)s^{-1},\,\,|s|=1,
\end{displaymath}
which forms a zero-pair of the kernel $U(x,y)=0$, we consider the following boundary value problem: Construct a simple smooth contour $\mathcal{L}$, a function $\lambda(z)$, $z\in\mathcal{L}$, and two functions $x(z)$, $y(z)$ such that
\begin{enumerate}
\item $z=0\in \mathcal{L}^{+}$, $z=1\in\mathcal{L}$, $z=\infty\in\mathcal{L}^{-}$,
\item $\lambda(z):\mathcal{L}\to[0,2\pi)$, $\lambda(1)=0$,
\item $x(z)$ is regular for $z\in\mathcal{L}^{+}$, continuous and univalent in $\mathcal{L}\cup\mathcal{L}^{+}$.
\item $y(z)$ is regular for $z\in\mathcal{L}^{-}$, continuous and univalent in $\mathcal{L}\cup\mathcal{L}^{-}$.
\item $x(z)$ has a simple zero at $z=0$, $x(1)=1$, $\lim_{z\to 0}\frac{x(z)}{z}>0$.
\item $y(z)$ has a simple zero at $z=\infty$, $y(1)=1$, $\lim_{|z|\to \infty}|zy(z)|>0$.
\item $x(z)$ (resp. $y(z)$) maps $\mathcal{L}^{+}$ (resp. $\mathcal{L}^{-}$) conformally to $S_{1}^{+}$ (resp. $S_{2}^{+}$).
\item For $z\in\mathcal{L}$, $(x^{+}(z),y^{-}(z))$ is a zero-pair of the kernel $U(x,y)=0$, where $x^{+}(z)=\lim_{t\to z,t\in\mathcal{L}^{+}}x(t)$, $y^{-}(z)=\lim_{t\to z,t\in\mathcal{L}^{-}}y(t)$.
\end{enumerate}

The solution of this boundary value problem depends on the position of $x=0$, $y=0$ with respect to $S_{1}$ and $S_{2}$; see \cite[II.3.10]{bv}. Without loss of generality, we only consider the first case, where $x=0\in S_{1}^{+},\,y=0\in S_{2}^{-}$. The others can be treated analogously (the second case is much more complicated; see \cite[II.3.12]{bv}). Following \cite[II.3.10-3.11]{bv}, and setting
\begin{displaymath}
x^{+}(z)=g(e^{\frac{1}{2}i\lambda(z)})e^{\frac{1}{2}i\lambda(z)},\,\,y^{-}(z)=g(e^{\frac{1}{2}i\lambda(z)})e^{-\frac{1}{2}i\lambda(z)},\,z\in\mathcal{L},
\end{displaymath} 
we obtain 
\begin{displaymath}
\begin{array}{rl}
x(z)=&ze^{\frac{1}{2\pi i}\int_{\zeta\in\mathcal{L}}[log\{\frac{g(e^{\frac{1}{2}i\lambda(z)})}{\sqrt{\zeta}}\}(\frac{\zeta+z}{\zeta-z}-\frac{\zeta+1}{\zeta-1})\frac{d\zeta}{\zeta}]},\,z\in\mathcal{L}^{+},\\
y(z)=&e^{-\frac{1}{2\pi i}\int_{\zeta\in\mathcal{L}}[log\{\frac{g(e^{\frac{1}{2}i\lambda(z)})}{\sqrt{\zeta}}\}(\frac{\zeta+z}{\zeta-z}-\frac{\zeta+1}{\zeta-1})\frac{d\zeta}{\zeta}]},\,z\in\mathcal{L}^{-}.
\end{array}
\end{displaymath}
The relation for the determination of $\mathcal{L}$ and $\lambda(z)$, $z\in\mathcal{L}$ is
\begin{equation}
e^{i\lambda(z)}=ze^{\frac{2}{2i\pi}\int_{\zeta\in\mathcal{L}}[log\{\frac{g(e^{\frac{1}{2}i\lambda(z)})}{\sqrt{\zeta}}\}(\frac{\zeta+z}{\zeta-z}-\frac{\zeta+1}{\zeta-1})\frac{d\zeta}{\zeta}]},\,z\in \mathcal{L}.
\label{int}
\end{equation}
An equivalent to \eqref{int} integral equations is
\begin{equation}
\frac{g(e^{\frac{1}{2}i\lambda(z)})}{\sqrt{z}}=e^{\frac{1}{4i\pi}\int_{\zeta\in\mathcal{L}}[(i\lambda(z)-log\{\zeta\})(\frac{\zeta+z}{\zeta-z}-\frac{\zeta+1}{\zeta-1})\frac{d\zeta}{\zeta}]},\,z\in \mathcal{L},\label{soo}
\end{equation}
with $\mathcal{L}=\{z:z=\rho(\phi)e^{i\phi},0\leq\phi\leq 2\pi\}$, $\theta(\phi)=\lambda(\rho(\phi)e^{i\phi})$. Separating real and imaginary parts in (\ref{soo}) we obtain two singular integral equations in the two unknowns functions $\rho(.)$, $\theta(.)$. Note that (\ref{soo}) may be regarded as a generalization of the Theodorsen's integral equation, see \cite[Section IV.2.3]{bv}.

Note finally that due to the maximum modulus principle \cite{neh}:
\begin{displaymath}
\begin{array}{lr}
|x(z)|<1,\,z\in\mathcal{L}^{+}\cup\mathcal{L},&
|y(z)|<1,\,z\in\mathcal{L}^{-}\cup\mathcal{L}.
\end{array}
\end{displaymath}
\subsection{Solution of the functional equation}
The following theorem is the main result of this section.
\begin{theorem}\label{thw}
Under stability condition \eqref{lm}, for $x\in S_{1}^{+}$, $y\in S_{2}^{+}$,
\begin{equation}
\begin{array}{rl}
\Pi_{0}(x,y)=&\frac{T(x,y)\Pi_{0}(0,0)}{\psi(x,y)-xy}\\
&\times\{\frac{K(x,y)}{2i\pi}\int_{\zeta\in\mathcal{L}}H(\zeta)[\frac{1}{x_{10}(x)-\zeta}-\frac{1}{y_{10}(x)-\zeta}]d\zeta+C(x,y)\},\vspace{2mm}\\
\Pi_{k}(x,y)=&F_{0,k}(x,y)\Pi_{0}(x,y),\,k=1,\ldots,N,\\
F_{0,k}(x,y)=&\frac{1}{det[\mathbf{L}(x,y)^{T}]}(cof \mathbf{L}(x,y)^{T})^{T}\mathbf{E}(x,y),
\end{array}\label{fins}
\end{equation}
where $x_{10}(.)$, $y_{10}(.)$ conformal mappings of $S_{1}$, $S_{2}$ onto the smooth and closed contour $\mathcal{L}$, respectively, $\Pi_{0}(0,0)$ as given in \eqref{lm}, and $H(z)=\frac{C(x^{+}(z),y^{-}(z))}{K(x^{+}(z),y^{-}(z))}$, $z\in\mathcal{L}$.
\end{theorem}
\begin{proof}
Since $(x^{+}(z),y^{-}(z))$, $z\in\mathcal{L}$ is a zero pair of the kernel, it should hold for $z\in\mathcal{L}$
\begin{equation}
\begin{array}{r}
K(x^{+}(z),y^{-}(z))[(1-w)\Pi_{0}(x^{+}(z),0)-w\Pi_{0}(0,y^{-}(z))]\\+C(x^{+}(z),y^{-}(z))\Pi_{0}(0,0)=0,
\end{array}
\end{equation}
or equivalently
\begin{equation}
\Pi_{0}(x^{+}(z),0)=\frac{w}{1-w}\Pi(0,y^{-}(z))-\frac{\Pi_{0}(0,0)}{1-w}H(z).
\label{pp}\end{equation}
Note that $\frac{w}{1-w}$ never vanishes and thus $index_{z\in\mathcal{L}}(\frac{w}{1-w})=0$. Moreover, $\frac{w}{1-w}$ satisfies (trivially) the Holder condition on $\mathcal{L}$. The numerator and the denominator of $H(.)$ both satisfy the Holder condition on $\mathcal{L}$, and the denominator never vanishes on $\mathcal{L}$, except for $z=1$. 

Therefore, we have the following non-homogeneous Riemann boundary value problem: 
\begin{enumerate}
\item $\Pi_{0}(x(z),0)$ should be regular for $z\in\mathcal{L}^{+}$ and continuous in $\mathcal{L}^{+}\cup \mathcal{L}$, with $\lim_{\zeta\to z,\zeta\in\mathcal{L}^{+}}\Pi_{0}(x(z),0)=\Pi_{0}(x^{+}(z),0)$, $z\in \mathcal{L}$,
\item $\Pi_{0}(0,y(z))$ should be regular for $z\in\mathcal{L}^{-}$ and continuous in $\mathcal{L}^{-}\cup \mathcal{L}$, with $\lim_{\zeta\to z,\zeta\in\mathcal{L}^{-}}\Pi_{0}(0,y(z))=\Pi_{0}(0,y^{-}(z))$, $z\in \mathcal{L}$,
\item For $z\in\mathcal{L}$, the boundary condition \eqref{pp} is satisfied.
\end{enumerate}
The solution of the above presented boundary value problem \cite{ga} is given by:
\begin{equation}
\begin{array}{rl}
(1-w)\Pi_{0}(x(z),0)=&\frac{\Pi_{0}(0,0)}{2i\pi}\int_{\zeta\in\mathcal{L}}H(\zeta)\frac{d\zeta}{z-\zeta}+w\Pi_{0}(0,0),\,z\in\mathcal{L}^{+},\\
w\Pi_{0}(0,y(z))=&\frac{\Pi_{0}(0,0)}{2i\pi}\int_{\zeta\in\mathcal{L}}H(\zeta)\frac{d\zeta}{z-\zeta}+w\Pi_{0}(0,0),\,z\in\mathcal{L}^{-},\end{array}
\end{equation}
where $\Pi_{0}(0,0)$ is given in \eqref{lm}. Denote by,
\begin{displaymath}
z=x_{10}(x),\,x\in S_{1}^{+},\,z=y_{10}(y),\,y\in S_{2}^{+},
\end{displaymath}
the inverse mappings, i.e., the conformal mappings of $S_{j}^{+}$ onto $\mathcal{L}$, $j=1,2,$ respectively. Since $S_{1}$, $S_{2}$, $\mathcal{L}$ are smooth, the theorem of corresponding boundaries implies that $x_{10}(.)$ maps $S_{1}$ onto $\mathcal{L}$, and $y_{10}(.)$ maps $S_{2}$ onto $\mathcal{L}$. Thus, we have
\begin{equation}
\begin{array}{rl}
(1-w)\Pi_{0}(x,0)=&\frac{\Pi_{0}(0,0)}{2i\pi}\int_{\zeta\in\mathcal{L}}H(\zeta)\frac{d\zeta}{x_{10}(x)-\zeta}+w\Pi_{0}(0,0),\,x\in S_{1}^{+},\\
w\Pi_{0}(0,y)=&\frac{\Pi_{0}(0,0)}{2i\pi}\int_{\zeta\in\mathcal{L}}H(\zeta)\frac{d\zeta}{y_{10}(y)-\zeta}+w\Pi_{0}(0,0),\,y\in S_{2}^{+},\end{array}\label{rbv}
\end{equation}
Substituting \eqref{rbv} in \eqref{fe}, it follows for $x\in S_{1}^{+}$, $y\in S_{2}^{+}$,
\begin{equation}
\begin{array}{rl}
\Pi_{0}(x,y)=&\frac{T(x,y)\Pi_{0}(0,0)}{\psi(x,y)-xy}\\
&\times\{\frac{K(x,y)}{2i\pi}\int_{\zeta\in\mathcal{L}}H(\zeta)[\frac{1}{x_{10}(x)-\zeta}-\frac{1}{y_{10}(y)-\zeta}]d\zeta+C(x,y)\}.
\end{array}\label{fin}
\end{equation}
By using analytic continuation arguments we can also obtain an expression for the $\Pi_{0}(x,y)$, for $|x|\leq 1$, $|y|\leq 1$. Using now \eqref{a2}, we obtain expressions for $\Pi_{k}(x,y)$, $k=1,\ldots,N$.
\end{proof}
\begin{remark}
Note that the analysis performed in Section \ref{sec:rbv} is general enough to be applied in the case where at least one of the following conditions are satisfied,  
\begin{displaymath}
\begin{array}{rl}
\frac{q_{-1,j}(0)}{q_{-1,j}^{(1)}(0)}+\frac{q_{i,-1}(0)}{q_{i,-1}^{(2)}(0)}\neq&1,\,i,j\in\mathbb{H}^{+}.
\end{array}
\end{displaymath}
However, some extra technical difficulties regarding the investigation of the properties of the curves $S_{1}$, $S_{2}$ will arise. Moreover, the resulting Riemann boundary value problem will be far more complicated, where at the same time the computation of its index is a challenging task. 
\end{remark}

\section{Explicit expressions of the moments for a special case}\label{symm}
In this section we provide explicit expressions for the moments of $\{Z(t)\}$ when we consider the \textit{symmetrical assumption} when the background state $J(t)=0$. Moreover, to enhance the readability we also focus on the QBD version of $\{Z(t), t\geq0\}$. More precisely, we consider the case where the increments of $\{X_{j}(t)\}$, $j=1,2,$ are in $\mathbb{H}=\{-1,0,1\}$ when $J(t)=0$, and in $\mathbb{H}^{+}=\{0,1\}$ when $J(t)=1,\ldots,N$. Note here that the analysis that follows can be also considered in the general case at a cost of more complicated expressions. Symmetry implies 
\begin{displaymath}
\begin{array}{c}
q_{-1,j}(0)=q_{j,-1}(0),\,q_{-1,j}^{(1)}(0)=q_{j,-1}^{(2)}(0),\,j\in\mathbb{H}^{+},\\
q_{0,1}(0)=q_{1,0}(0),\,\theta_{0,k}=\theta,\,A_{0,k}(x,y)=A(x,y),\,k=1,\ldots,N.
\end{array}
\end{displaymath} 
and $w=1/2$, $\Pi_{0}(1,0)=\Pi_{0}(0,1)$, $\Pi_{0}^{(1)}(1,1)=\Pi_{0}^{(2)}(1,1)$, where $\Pi_{0}^{(k)}(1,1)$, $k=1,2$ the derivatives of $\Pi_{0}(x,y)$ with respect to $x$ and $y$, respectively, at $(1,1)$. Note that under symmetry assumptions, the stability condition \eqref{lm} can be written after some algebra as
\begin{displaymath}
\rho:=\frac{T(1,1)(q_{1,0}(0)+q_{1,1}(0))}{T(1,1)(N\theta+q_{-1,0}(0))+N\theta T^{(1)}(1,1)-1}<1,
\end{displaymath}
and
\begin{displaymath}
\begin{array}{l}
\Pi_{0}(0,0)=\frac{2\Pi_{0}(1,1)[N\theta(T^{(1)}(1,1)+T(1,1))+T(1,1)(q_{-1,0}(0)-q_{1,0}(0)-q_{1,1}(0))-1]}{T(1,1)q_{0,-1}^{(2)}(0)},
\end{array}
\end{displaymath}
where $T^{(1)}(1,1)$ the derivative of $T(x,y)$ with respect to $x$ at $(1,1)$. Let $M_{m}=E(X_{m})$, $m=1,2,$ the first moment of the stationary distribution of $\{X_{m}(t);t\geq0\}$. The following theorem provides the main result of this section.
\begin{theorem}\label{th2}
When $\rho<1$,
\begin{equation}
\begin{array}{rl}
M_{1}=&\Pi_{0}(1,1)\sum_{k=1}^{N}\frac{\partial}{\partial x}F_{0,k}(x,1)|_{x=1}+M(1+\sum_{k=1}^{N}F_{0,k}(1,1)),\vspace{2mm}\\
M_{2}=&\Pi_{0}(1,1)\sum_{k=1}^{N}\frac{\partial}{\partial y}F_{0,k}(1,y)|_{y=1}+M(1+\sum_{k=1}^{N}F_{0,k}(1,1))
\end{array}\label{sym0}
\end{equation}
where
\begin{equation}
\begin{array}{rl}
M:=&\frac{\Pi_{0}(0,0)q_{0,-1}^{(2)}(0)[2T^{(1)}(1,1)+T(1,1)]+(S/2)\Pi_{0}(1,1)}{2[N\theta(T^{(1)}(1,1)+T(1,1))+T(1,1)(q_{-1,0}(0)-q_{1,0}(0)-q_{1,1}(0))-1]},\vspace{2mm}\\
S:=&2\left[2T^{(1)}(1,1)(N\theta+q_{-1,0}(0)-q_{1,0}(0)-q_{1,1}(0))+T(1,1) \right.\\&\left.\times(N\theta+2q_{-1,0}(0)-4q_{1,0}(0)-5q_{1,1}(0))\right]+N\theta\frac{\partial^{2}}{\partial x^{2}}T(x,x)|_{x=1},
\end{array}\label{sym2}
\end{equation}
and $\Pi_{0}(1,1)$ as given in \eqref{prob}.
\end{theorem}
\begin{proof}
Note that
\begin{equation}
\begin{array}{rl}
M_{1}:=&\sum_{k=0}^{N}\frac{\partial}{\partial x}\Pi_{k}(x,1)|_{x=1}=\sum_{k=0}^{N}\Pi_{k}^{(1)}(1,1)\\
=&\Pi_{0}^{(1)}(1,1)+\sum_{k=1}^{N}[\Pi_{0}(1,1)\frac{\partial}{\partial x}F_{0,k}(x,1)|_{x=1}+F_{0,k}(1,1)\Pi_{0}^{(1)}(1,1)]\\
=&\Pi_{0}^{(1)}(1,1)[1+\sum_{k=1}^{N}F_{0,k}(1,1)]+\Pi_{0}(1,1)\sum_{k=1}^{N}\frac{\partial}{\partial x}F_{0,k}(x,1)|_{x=1}.
\end{array}\label{sym3}
\end{equation}

A similar expression can also be derived for $E(X_{2})$. Thus, we only need to find an expression for the $\Pi_{0}^{(1)}(1,1)$. Setting $x=y$, we realise that $K(x,x)=0$, and thus, \eqref{fe} is rewritten as
\begin{equation}
\Pi_{0}(x,x)\frac{T(x,x)R(x,x)-x^{2}}{x-1}=T(x,x)xq_{0,-1}^{(2)}(0)\Pi_{0}(0,0).
\label{sym1}
\end{equation}
By differentiating \eqref{sym1} with respect to $x$ at $x=1$, applying once the L'Hospital rule, and having in mind the symmetry condition at phase $0$, we obtain
\begin{displaymath}
\Pi_{0}^{(1)}(1,1)=\Pi_{0}^{(2)}(1,1)=M,
\end{displaymath}
as given in the first in \eqref{sym2}. Substituting back in \eqref{sym3} we obtain the first in \eqref{sym0}. Similarly we obtain the expression for $M_{2}$.
\end{proof}
\section{Application: A retrial system with a finite capacity priority line and two coupled orbit queues}\label{appl}
In this section, we show that this methodological approach can be also used to provide the stationary analysis of a single server retrial system with a finite priority queue of capacity $N$, and two coupled orbits of infinite capacity. We assume that three classes of jobs arrive at the system, say $P_{k}$, $k=0,1,2$.

$P_{k}$ jobs arrive according to a Poisson process with rate $\lambda_{k}$. Arriving jobs that find the server idle, they get service immediately. If an arriving $P_{0}$ job finds the server busy, waits in an ordinary queue provided that there is available space. If an arriving $P_{0}$ job finds the queue fully occupied is considered lost. 

An arriving $P_{k}$, $k=1,2$ job that finds the server busy, it is routed in orbit queue $k$. Orbiting jobs of either type retry independently to connect with the server according to a state-dependent constant retrial policy. More precisely, when both orbit queues are non-empty, the orbit queue 1 (resp. 2) attempts to re-dispatch a blocked type $P_{1}$ job (resp. $P_{2}$) to the service
station after an exponentially distributed time with rate $w\alpha_{1}$ (resp. $\bar{w}\alpha_{2}$, where $\bar{w}=1-w$). If orbit queue 1 (resp. 2)
is the only non-empty, then, it changes its re-dispatch rate to $\alpha_{1}$ (resp. $\alpha_{2}$). Service time is independent of the job type of and it is exponentially distributed with rate $\mu$.

Under such a setting, the level process $\{(X_{1}(t),X_{2}(t))\}$ taking values in $\mathbb{Z}^{+}_{2}$, corresponds to the orbit queue lengths at time $t$, while the phase process $\{J(t)\}$ corresponds to the number of jobs in the priority queue and in the service station at time $t$, taking value in $S_{0}=\{0,1,\ldots,N\}$. Note that in such a case, $\{Z(t)\}$ corresponds to a two-dimensional quasi birth-death (QBD) process with a similar structure as the one discussed in Section \ref{sec:mod}.

Writing the balance equations and using the generating function approach we come up with the system of equations
\begin{equation}
\begin{array}{rl}
\widehat{\alpha}\Pi_{0}(x,y)-\mu\Pi_{1}(x,y)=&(\alpha_{2}-\alpha_{1})[\bar{w}\Pi_{0}(x,0)-w\Pi_{0}(0,y)]\\+(\alpha_{1}w+\alpha_{2}\bar{w})\Pi_{0}(0,0),
\end{array}\label{g1}
\end{equation}
\begin{equation}
\begin{array}{l}
\Pi_{0}(x,y)u_{0}(x,y)-(\lambda+\mu)\Pi_{1}(x,y)+\mu xy\Pi_{2}(x,y)\\
=(\alpha_{2}x-\alpha_{1}y)[w\Pi_{0}(x,0)-\bar{w}\Pi_{0}(0,y)]+(\alpha_{1}\bar{w}y+\alpha_{2}wx)\Pi_{0}(0,0),
\end{array}\label{g2}
\end{equation}
\begin{equation}
\begin{array}{rl}
u_{1}(x,y)\Pi_{k}(x,y)=&\lambda_{0}\Pi_{k-1}(x,y)+\mu\Pi_{k+1}(x,y),\,k=2,\ldots,N-1,\vspace{2mm}\\
u_{2}(x,y)\Pi_{N}(x,y)=&\lambda_{0}\Pi_{n-1}(x,y),
\end{array}\label{g3}
\end{equation}
where, $\lambda=\lambda_{0}+\lambda_{1}+\lambda_{2}$, $\widehat{\alpha}=\lambda+\alpha_{1}w+\alpha_{2}(1-w)$, $u_{0}(x,y)=\lambda xy+w\alpha_{1}y+(1-w)\alpha_{2}x$, $u_{1}(x,y)=\mu+\lambda_{0}+\lambda_{1}(1-x)+\lambda_{2}(1-y)$, $u_{2}(x,y)=\mu+\lambda_{1}(1-x)+\lambda_{2}(1-y)$.

Equations \eqref{g1}-\eqref{g3} give rise to the following \textit{matrix} functional equation, which has the same form as the one in \eqref{matrix}, where now,
\begin{displaymath}
\begin{array}{rl}
\mathbf{Q}(x,y)=&\begin{pmatrix}
\widehat{\alpha}&u_{0}(x,y)&0&\\
-\mu&-(\lambda+\mu)xy&\lambda_{0}&0\\
0&\mu xy&-u_{1}(x,y)&\lambda_{0}&0\\
0&0&\mu&-u_{1}(x,y)&\lambda_{0}&0&\\
\vdots&\vdots&\vdots&\ddots&\ddots\\
0&0&\ldots&&\ldots&\mu&-u_{1}(x,y)&\lambda_{0}\\
0&0&\ldots&&\ldots&&\mu&-u_{2}(x,y)
\end{pmatrix},\vspace{2mm}\\
\mathbf{T}_{1}(x,y)=&\begin{pmatrix}
\alpha_{2}-\alpha_{1}&\alpha_{2}x-\alpha_{1}y&\ldots&0\\
0&0&\ldots&0\\
\vdots&\vdots&\ddots&\vdots\\
0&0&\ldots&0\\
\end{pmatrix},\\
\mathbf{T}_{2}(x,y)=&\begin{pmatrix}
\alpha_{1}\bar{w}+\alpha_{2}w&\alpha_{1}\bar{w}y+\alpha_{2}wx&\ldots&0\\
0&0&\ldots&0\\
\vdots&\vdots&\ddots&\vdots\\
0&0&\ldots&0\\
\end{pmatrix},
\end{array}
\end{displaymath}
Note here that compared with the general framework discussed in Section \ref{sec:fun}, the matrix $\mathbf{Q}(x,y)$ is triangular due to the fact that according to the model description, we allow transitions only among neighbour states of the phase process. Moreover, the structure of matrices $\mathbf{T}_{1}(x,y)$, $\mathbf{T}_{2}(x,y)$ are slightly different compared with those in Section \ref{sec:fun}. 

Using \eqref{g3}, $\Pi_{k}(x,y)$, $k=2,\ldots,N$ is written in terms of $\Pi_{1}(x,y)$. This procedure yields after some algebra
\begin{equation}
\Pi_{k}(x,y)=s_{k-1}(x,y)\Pi_{k-1}(x,y),\,k=2,3,\ldots,N,\label{dlp}
\end{equation}
where $s_{N}(x,y)=0$ and for $k=1,\ldots,N-1,$ we can recursively obtain
\begin{equation}
s_{k}(x,y)=\frac{\lambda_{0}}{u_{2}(x,y)+\lambda_{0}1_{\{k\neq N-1\}}-\mu s_{k+1}(x,y)}.\label{lmk}
\end{equation}
Substituting \eqref{dlp} in \eqref{g2}, and then using \eqref{g1}, we obtain
\begin{equation}
\begin{array}{rl}
\tilde{R}(x,y)\Pi_{0}(x,y)=\tilde{K}(x,y)[\bar{w}\Pi_{0}(x,0)-w\Pi_{0}(0,y)]+\tilde{C}(x,y)\Pi_{0}(0,0),
\end{array}\label{fe1}
\end{equation}
where now,
\begin{displaymath}
\begin{array}{rl}
\tilde{R}(x,y)=&\widehat{\alpha}xy[\mu(1+s_{1}(x,y))-u_{1}(x,y)]+w\alpha_{1}\mu y(1-x)+\bar{w}\alpha_{2}\mu x(1-y),\\
\tilde{K}(x,y)=&\alpha_{2}\mu x(1-y)-\alpha_{1}\mu y(1-x)+(\alpha_{2}-\alpha_{1})xy[\mu(1+s_{1}(x,y))-u_{1}(x,y)],\\
\tilde{C}(x,y)=&\bar{w}\alpha_{1}y[(1-x)(\mu-\lambda_{1}x)+x(\mu(1+s_{1}(x,y))-u_{1}(x,y))]\\
&+w\alpha_{2}x[(1-y)(\mu-\lambda_{2}y)+y(\mu(1+s_{1}(x,y))-u_{1}(x,y))].
\end{array}
\end{displaymath}

Note that \eqref{fe1} has exactly the same form as the one given in \eqref{fe}. Thus, using the PSA or the theory of boundary value problems presented in sections \ref{psa} and \ref{sec:rbv}, respectively, we are able to solve it, and obtain expressions for the $\Pi_{0}(x,y)$. Then, equation \eqref{g1} can be used to obtain expressions for the $\Pi_{1}(x,y)$. Having obtain $\Pi_{1}(x,y)$, the rest of the unknown functions, $\Pi_{k}(x,y)$ are obtained recursively by using \eqref{dlp}.

We conclude this section by obtaining the stationary probabilities of phase process, and the stability condition. Note that starting from \eqref{dlp}, we have
\begin{displaymath}
\Pi_{k}(x,y)=\Pi_{1}(x,y)\prod_{j=1}^{k-1}s_{j}(x,y),\,k=2,\ldots,N.
\end{displaymath}
Then, using \eqref{lmk}, and setting $(x,y)=(1,1)$, we realize that $s_{k}(1,1)=\frac{\lambda_{0}}{\mu}$, $k=1,\ldots,N-1$. Therefore, the normalization condition reads
\begin{equation}
\begin{array}{rl}
1=&\Pi_{0}(1,1)+\Pi_{1}(1,1)\sum_{k=0}^{N-1}(\frac{\lambda_{0}}{\mu})^{k}=\Pi_{0}(1,1)+\Pi_{1}(1,1)\frac{1-\rho^{N}_{0}}{1-\rho_{0}},
\end{array}\label{norm}
\end{equation}
where $\rho_{j}=\lambda_{j}/\mu$, $j=0,1,2.$ For each $i = 0, 1,...$, consider the vertical cut between the states $\{X_{1} = i, J = 1\}$ and $\{X_{1} = i + 1,J = 0\}$. Then, if $\pi^{(k)}_{i,.}=\sum_{j=0}^{\infty}\pi^{(k)}_{i,j}$, $k\in S_{0}$ we have
\begin{displaymath}
\begin{array}{rl}
\lambda_{1}\pi^{(1)}_{i,.}=&\alpha_{1}\pi^{(0)}_{i+1,0}+w\alpha_{1}\sum_{j=1}^{\infty}\pi^{(0)}_{i+1,j}\alpha_{1}\pi^{(0)}_{i+1,0}+w\alpha_{1}[\pi^{(0)}_{i+1,.}-\pi^{(0)}_{i+1,0}]\\
=&\bar{w}\alpha_{1}\pi^{(0)}_{i+1,0}+w\alpha_{1}\pi^{(0)}_{i+1,.}.
\end{array}
\end{displaymath}
Summing for all $i\geq 0$ yields,
\begin{equation}
\begin{array}{c}
\lambda_{1}\Pi_{1}(1,1)=\bar{w}\alpha_{1}[\Pi_{0}(1,0)-\Pi_{0}(0,0)]+w\alpha_{1}[\Pi_{0}(1,1)-\Pi_{0}(0,1)].
\end{array}\label{c1}
\end{equation}
Similar arguments can be used to obtain
\begin{equation}
\begin{array}{c}
\lambda_{2}\Pi_{1}(1,1)=w\alpha_{2}[\Pi_{0}(0,1)-\Pi_{0}(0,0)]+\bar{w}\alpha_{2}[\Pi_{0}(1,1)-\Pi_{0}(1,0)].
\end{array}\label{c2}
\end{equation}
Summing \eqref{c1}, \eqref{c2}, and \eqref{g1} after setting $(x,y)=(1,1)$ yields
\begin{displaymath}
\lambda\Pi_{0}(1,1)=(\mu-\lambda_{1}-\lambda_{2})\Pi_{1}(1,1).
\end{displaymath}
Using the normalization condition \eqref{norm}, the last expression is rewritten after some algebra as
\begin{displaymath}
\Pi_{1}(1,1)=\rho\frac{1-\rho_{0}}{\rho_{0}[1-\rho(1-\rho^{N-1}_{0})]},
\end{displaymath} 
where $\rho=\rho_{0}+\rho_{1}+\rho_{2}$. Thus,
\begin{displaymath}
\begin{array}{rl}
\Pi_{0}(1,1)=&1-\Pi_{1}(1,1)\frac{1-\rho^{N}_{0}}{1-\rho_{0}}=1-\rho\frac{1-\rho^{N}_{0}}{\rho_{0}[1-\rho(1-\rho^{N-1}_{0})]},\\
\Pi_{k}(1,1)=&\rho^{k-1}_{0}\Pi_{1}(1,1),\,k=2,3,\ldots,N.
\end{array}
\end{displaymath}
It remains to obtain $\Pi_{0}(0,0)$, i.e., the probability of an empty system. Multiply \eqref{c1} with $\alpha_{2}$, and \eqref{c2} with $\alpha_{1}$. By summing the resulting equations we obtain
\begin{displaymath}
\Pi_{0}(0,0)=\Pi_{0}(1,1)-\Pi_{1}(1,1)(\frac{\lambda_{1}}{\alpha_{1}}+\frac{\lambda_{2}}{\alpha_{2}}).
\end{displaymath}
Substituting $\Pi_{0}(1,1)$ and $\Pi_{1}(1,1)$ yields,
\begin{displaymath}
\begin{array}{rl}
\Pi_{0}(0,0)=&1-\frac{\rho}{\rho_{0}[1-\rho(1-\rho^{N-1}_{0})]}\left(1-\rho^{N}_{0}+(1-\rho_{0})(\frac{\lambda_{1}}{\alpha_{1}}+\frac{\lambda_{2}}{\alpha_{2}})\right)\\
=&1-\widehat{\rho}>0,
\end{array}
\end{displaymath}
where 
\begin{displaymath}
\widehat{\rho}=\frac{\rho}{\rho_{0}[1-\rho(1-\rho^{N-1}_{0})]}\left(1-\rho^{N}_{0}+(1-\rho_{0})(\frac{\lambda_{1}}{\alpha_{1}}+\frac{\lambda_{2}}{\alpha_{2}})\right)<1,
\end{displaymath}
is the stability condition for our system.
\section{Numerical results}\label{sec:num}
In the following, we present some numerical illustration of the theoretical results presented in Section \ref{psa}. In particular, we consider a two-node queueing network with coupled processors and two types of service interruption. Two classes of jobs arrive according to independent Poisson processes. Class $P_{i}$, $i=1,2$ is routed to queue $i$. Each job at queue $i$ requires exponentially distributed service time with rate $\nu_{i}$. The network is operating for an exponentially distributed time with rate $\theta_{0,j}=\gamma_{j}$, and then switches to the fail mode $j$, $j=1,2$. In failed mode $j$, it stays for an exponentially distributed time and then, it switches to operating mode 0 with rate $\theta_{j,0}=\tau_{j}$. When the network is in a failed mode of either type it cannot provide service.

Jobs arrival rates are $\lambda_{i}^{(j)}$, $i=1,2$, $j=0,1,2$, i.e., depend on the type of the mode (either operating or failed). When the network is in operational mode, and both queues are non-empty, queue 1 serves at a rate $w\nu_{1}$, and queue 2 at a rate $(1-w)\nu_{2}$. If only one queue is non-empty it serves at full capacity, i.e., with rate $\nu_{i}$. Upon receiving service at queue 1 (resp. 2) the job is either routed to queue 2 (resp. 1) with probability $r_{12}$ (resp. $r_{21}$), or leaves the network with probability $1-r_{12}$ (resp. $1-r_{21}$). For ease of computations we assume $a_{0,0}^{(k,m)}=1$, $k\neq m$, i.e., the change in mode does not change the number of jobs in queues. Thus, in such a case, $\{Z(t)\}$ is a QBD process, i.e., the increments of $\{X(t)\}$ are in $\mathbb{H}=\{-1,0,1\}$ when $J(t)=0$, and in $\mathbb{H}^{+}=\{0,1\}$ when $J(t)=1,2$.

Note also that in such a case, the stability condition \eqref{lm} reads \footnote{Note that $\frac{\tau_{1}\tau_{2}}{\tau_{1}\tau_{2}+\gamma_{1}\tau_{2}+\gamma_{2}\tau_{1}}\left(\frac{\Lambda^{(0)}_{1}}{\nu_{1}}+\frac{\Lambda^{(0)}_{2}}{\nu_{2}}\right)$ (resp. $\frac{\gamma_{k}\tau_{(k+1)mod 2}}{\tau_{1}\tau_{2}+\gamma_{1}\tau_{2}+\gamma_{2}\tau_{1}}\left(\frac{\Lambda^{(k)}_{1}}{\nu_{1}}+\frac{\Lambda^{(k)}_{2}}{\nu_{2}}\right)$) refers to the amount of work that arrive at the system per time unit when the network is in the operating mode (resp. in the failed mode $k=1,2$), while a job can depart from the network only when it is in the operating mode, i.e., with probability $\frac{\tau_{1}\tau_{2}}{\tau_{1}\tau_{2}+\gamma_{1}\tau_{2}+\gamma_{2}\tau_{1}}$.}
\begin{equation}
\begin{array}{l}
\frac{\tau_{1}\tau_{2}}{\tau_{1}\tau_{2}+\gamma_{1}\tau_{2}+\gamma_{2}\tau_{1}}\left(\frac{\Lambda^{(0)}_{1}}{\nu_{1}}+\frac{\Lambda^{(0)}_{2}}{\nu_{2}}\right)+\frac{\gamma_{1}\tau_{2}}{\tau_{1}\tau_{2}+\gamma_{1}\tau_{2}+\gamma_{2}\tau_{1}}\left(\frac{\Lambda^{(1)}_{1}}{\nu_{1}}+\frac{\Lambda^{(1)}_{2}}{\nu_{2}}\right)\\+\frac{\tau_{1}\gamma_{2}}{\tau_{1}\tau_{2}+\gamma_{1}\tau_{2}+\gamma_{2}\tau_{1}}\left(\frac{\Lambda^{(2)}_{1}}{\nu_{1}}+\frac{\Lambda^{(2)}_{2}}{\nu_{2}}\right)<\frac{\tau_{1}\tau_{2}}{\tau_{1}\tau_{2}+\gamma_{1}\tau_{2}+\gamma_{2}\tau_{1}},
\end{array}\label{sta1}
\end{equation}
where $\Lambda^{(k)}_{1}=\frac{\lambda_{1}^{(k)}+\lambda_{2}^{(k)}r_{21}}{1-r_{12}r_{21}}$, $\Lambda^{(k)}_{2}=\frac{\lambda_{2}^{(k)}+\lambda_{1}^{(k)}r_{12}}{1-r_{12}r_{21}}$, $k=0,1,2$. Note that \eqref{sta1} has a clear probabilistic interpretation, since the left hand side in (\ref{sta1}) equals the amount of work brought into the system per time unit, and in order the system to be stable, should be less than the amount of work departing the system per time unit, provided that the network is in the operating mode. 
\subsection{Numerical validation \& Influence of system parameters as $w\to0$}\label{z1}
Figure \ref{fig:7} depicts the approximations (\ref{pert}) as a function of $w$ for increasing values of $M$, under the set-up given in Table \ref{tab1}. The horizontal line (i.e., $M = 0$) refers to the special case where the second queue has priority over the first one. As expected, Figure \ref{fig:7} confirms that the PSA approximations are accurate when $w$ is close to 0, and clearly, by adding more terms, we can have larger regions for $w$ where the accuracy is good. 
\begin{table}[ht!]
\centering
\caption{Set-up.}
\label{tab1}
	{\large\begin{tabular}{|c|c|}
		\hline
		Parameters &Values\\
	\hline\hline
		$(\lambda^{(1)}_{1},\lambda^{(1)}_{2})$&$(0.5,0.6)$\\
		\hline
		$(\lambda^{(2)}_{1},\lambda^{(2)}_{2})$&$(0.1,0.2)$\\
		\hline
		$(\nu_{1},\nu_{2})$&$(5,6)$\\
		\hline
		$(\tau_{1},\tau_{2})$&$(5,8)$\\
		\hline
		$(\gamma_{1},\gamma_{2})$&$(0.5,0.8)$\\
		\hline
		$(r_{12},r_{21})$&$(0.3,0.2)$\\
		\hline
	\end{tabular}}	
\end{table}
 
\begin{figure*}
  \includegraphics[width=0.75\textwidth]{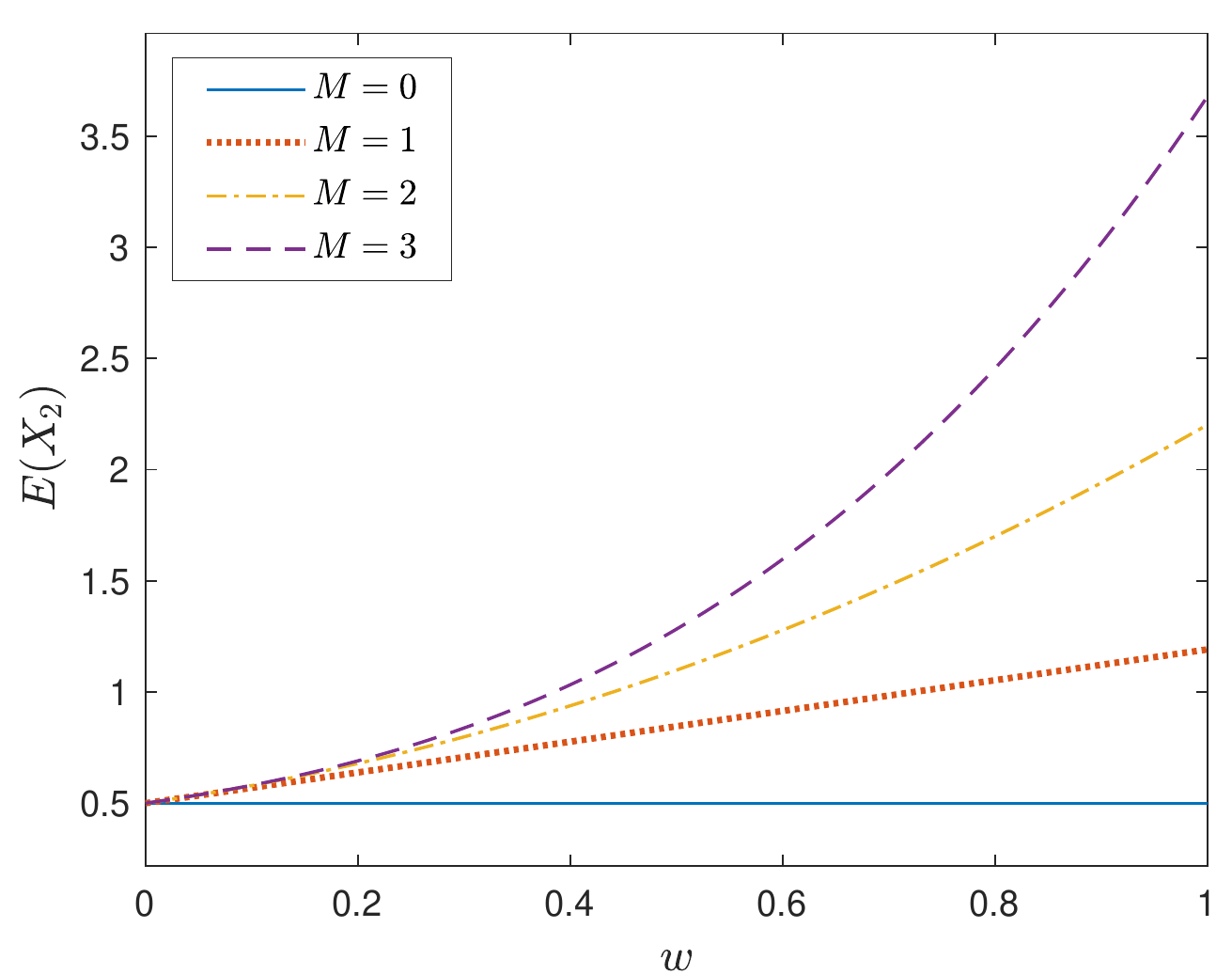}
\caption{Truncation approximation for $(\lambda^{(0)}_{1},\lambda^{(0)}_{2})=(1,0.8)$, $(\gamma_{1},\gamma_{2})=(0.5,0.8)$.}
\label{fig:7}       
\end{figure*}

In Figure \ref{fig:71}, we focus on the influence of system parameters on $E(X_{2})$ for the near priority system as $w\to0$. It is seen that as $\gamma_{2}$ increases the mean queue content in station 2 increases, and that increase becomes more apparent as $\lambda^{(0)}_{1}$, $\lambda^{(0)}_{2}$ increase too. This is expected since by increasing the failure rate and the load in system, the station 1 has always customers. Thus, sharing the server with that
station, even for a small percentage of the time can have a large influence, especially when the rate of failures increases too. Similar behaviour is expected if we fix $\gamma_{2}$, and let $\gamma_{1}$ to vary. 
\begin{figure*}
  \includegraphics[width=0.75\textwidth]{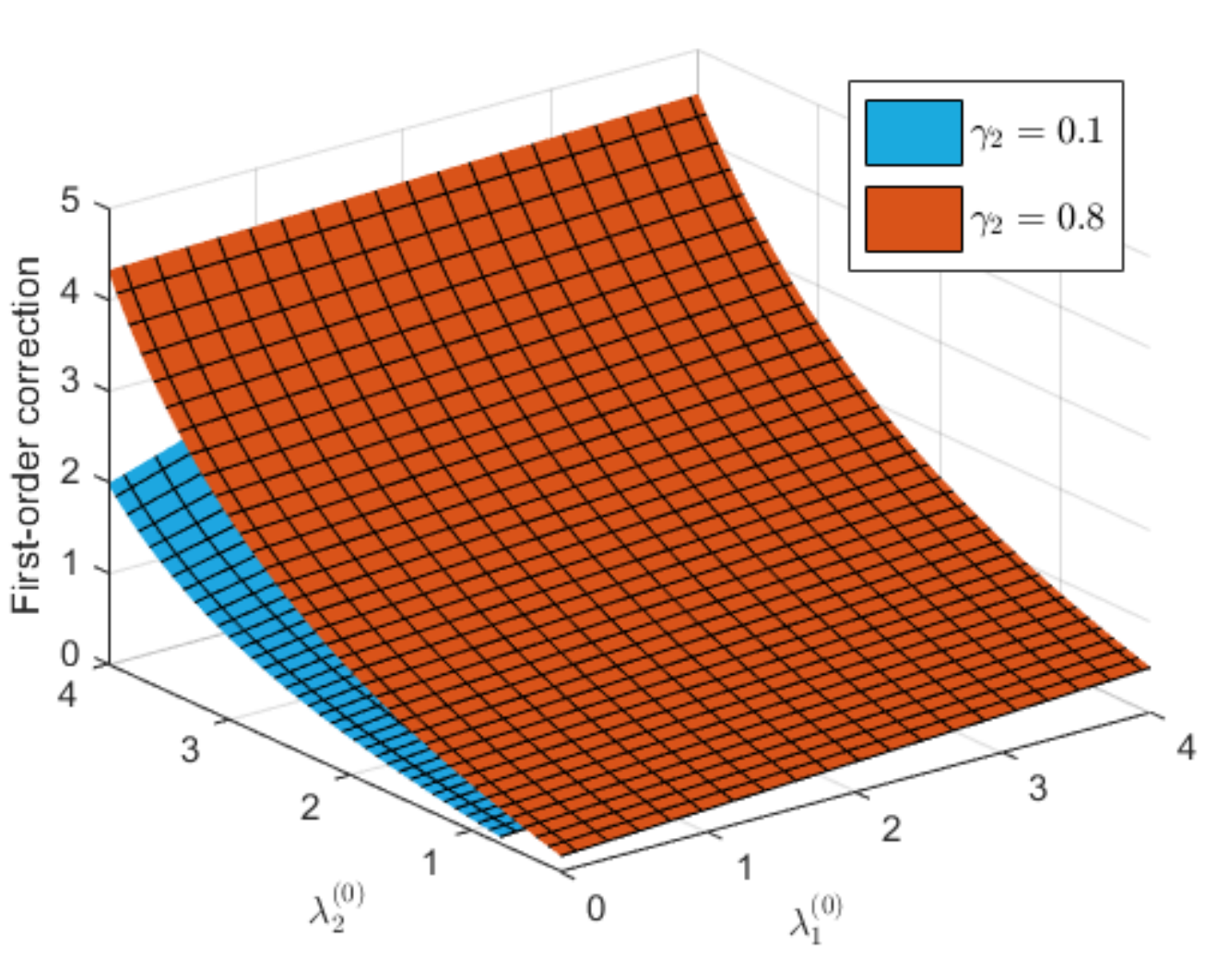}
\caption{First-order correction as $w\to 0$, $\gamma_{1}=0.5$.}
\label{fig:71}       
\end{figure*}
\subsection{The symmetrical model}\label{z2}
We now focus on the system analysed in Section \ref{symm}, and set $\lambda_{1}^{(0)}=\lambda_{2}^{(0)}:=\lambda$, $\gamma_{1}=\gamma_{2}:=\gamma$ and $w=1/2$. The rest of the parameter values are the same as those considered in subsection \ref{z1}, i.e., we have the following set-up:
\begin{table}[ht!]
\centering
\caption{Set-up.}
\label{tab2}
	{\large\begin{tabular}{|c|c|}
		\hline
		Parameters &Values\\
	\hline\hline
	$(\lambda^{(1)}_{1},\lambda^{(1)}_{2})$&$(0.5,0.6)$\\
		\hline
		$(\lambda^{(2)}_{1},\lambda^{(2)}_{2})$&$(0.1,0.2)$\\
		\hline
		$r_{12}=r_{21}=r$&$0.3$\\
		\hline
	\end{tabular}}	
\end{table}

Figure \ref{fig:81} depicts the total number of jobs in the system, i.e., $E(X_{1}+X_{2})$, as a function of $\lambda$ and $\gamma$. We observe sharp increase in $E(X_{1}+X_{2})$ when $\gamma$ take small values for increasing $\lambda$. This is expected since by increasing $\gamma$, the system is most of the time in failed modes, and since the parameters are considered fixed in these modes, even if we increase $\lambda$ (i.e., the arrival rate when the network is in operational mode), it will not cause significant effect. However, when $\gamma$ takes small values, the network remains in operational mode more time, and thus, the effect of $\lambda$ is crucial (as well as of service rate $\nu$). 

Similar behaviour is observed in Figure \ref{fig:11}, where we observe the way the total expected number of jobs in the system varies as a function of $\tau_{1}$, $\tau_{2}$ for different values of $\gamma$. Note that $E(X_{1}+X_{2})$ is very sensitive when $\tau_{j}$, $j=1,2$ takes small values. This is expected since small values of $\tau_{j}$ result in longer time periods for the network to return to the operation mode.

Figure \ref{fig:10} shows the way $E(X_{1}+X_{2})$ varies for increasing values of $\gamma$ and $\lambda$, as we increase $\tau_{1}$ from 5 to 15. This means that we considerably reduce the recovery time when the network is in failed mode 1. We can see that $E(X_{1}+X_{2})$ is considerably reduced, which is expected since the network returns faster in the operational mode. 

In Figure \ref{fig:9} we provide a comparison between the explicit expression on the expected number of jobs in queue 2 (see the second in \eqref{sym0}), and the one derived using the PSA for $M=10$ and $M=20$. We observe that when $\lambda$ takes small values, PSA provides very accurate results. As $\lambda$ increases, PSA accuracy is improved by providing further terms.
\begin{figure*}
  \includegraphics[width=0.75\textwidth]{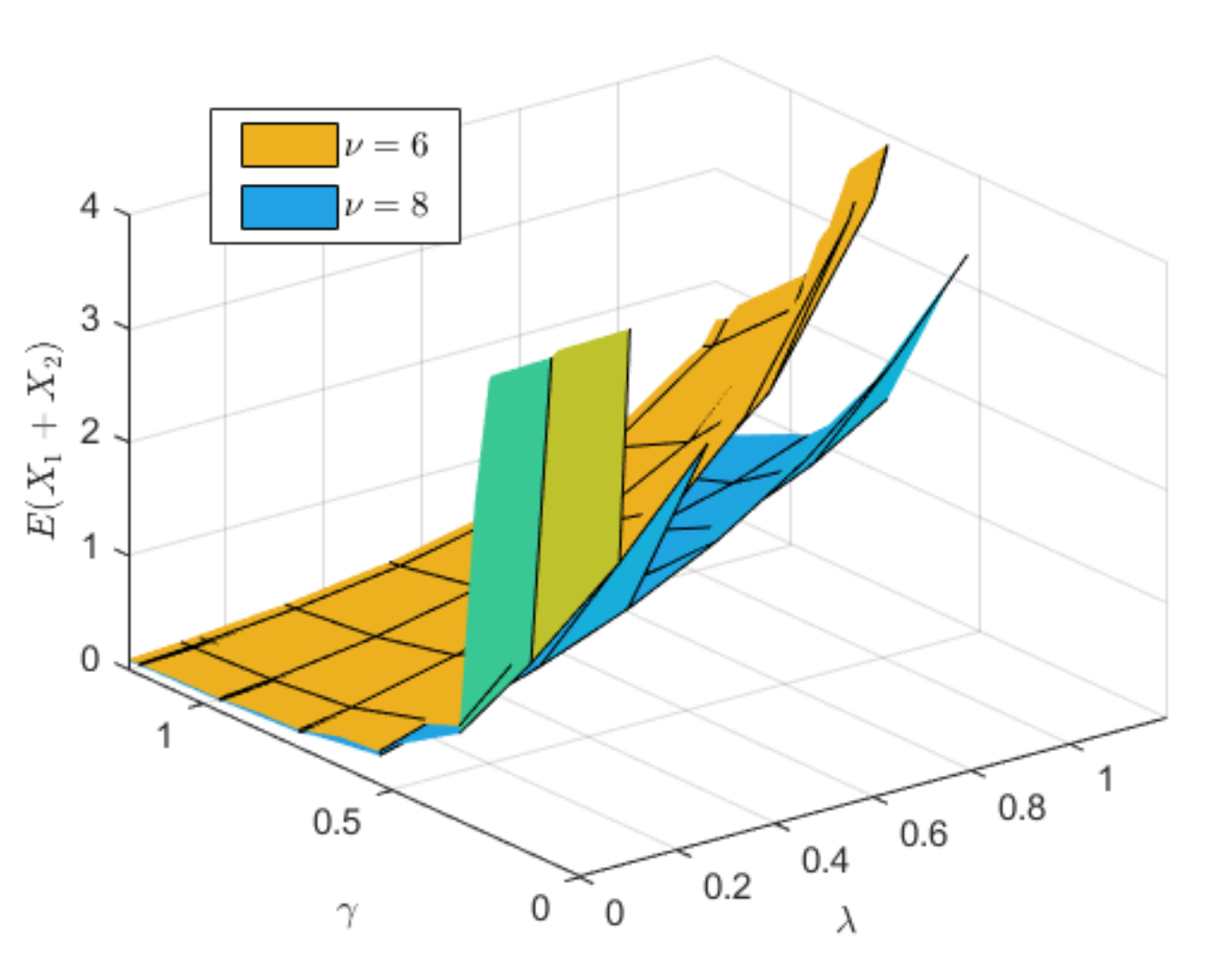}
\caption{$E(X_{1}+X_{2})$ for the symmetrical system for $(\tau_{1},\tau_{2})=(5,8)$.}
\label{fig:81}       
\end{figure*}\begin{figure*}
  \includegraphics[width=0.75\textwidth]{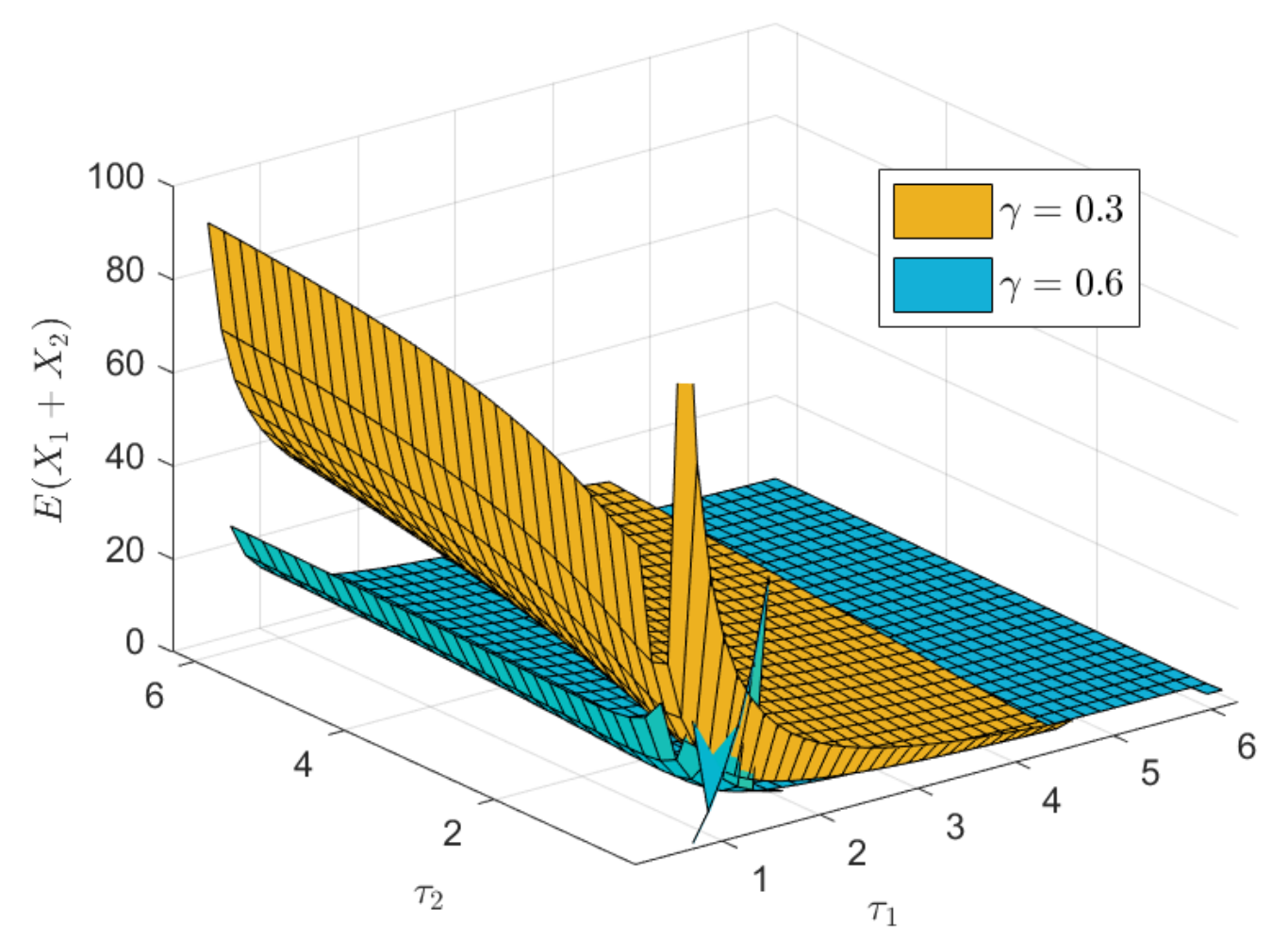}
\caption{$E(X_{1}+X_{2})$ for the symmetrical system and $\lambda=1.3$, $\nu=6$.}
\label{fig:11}       
\end{figure*}

\begin{figure*}
  \includegraphics[width=0.75\textwidth]{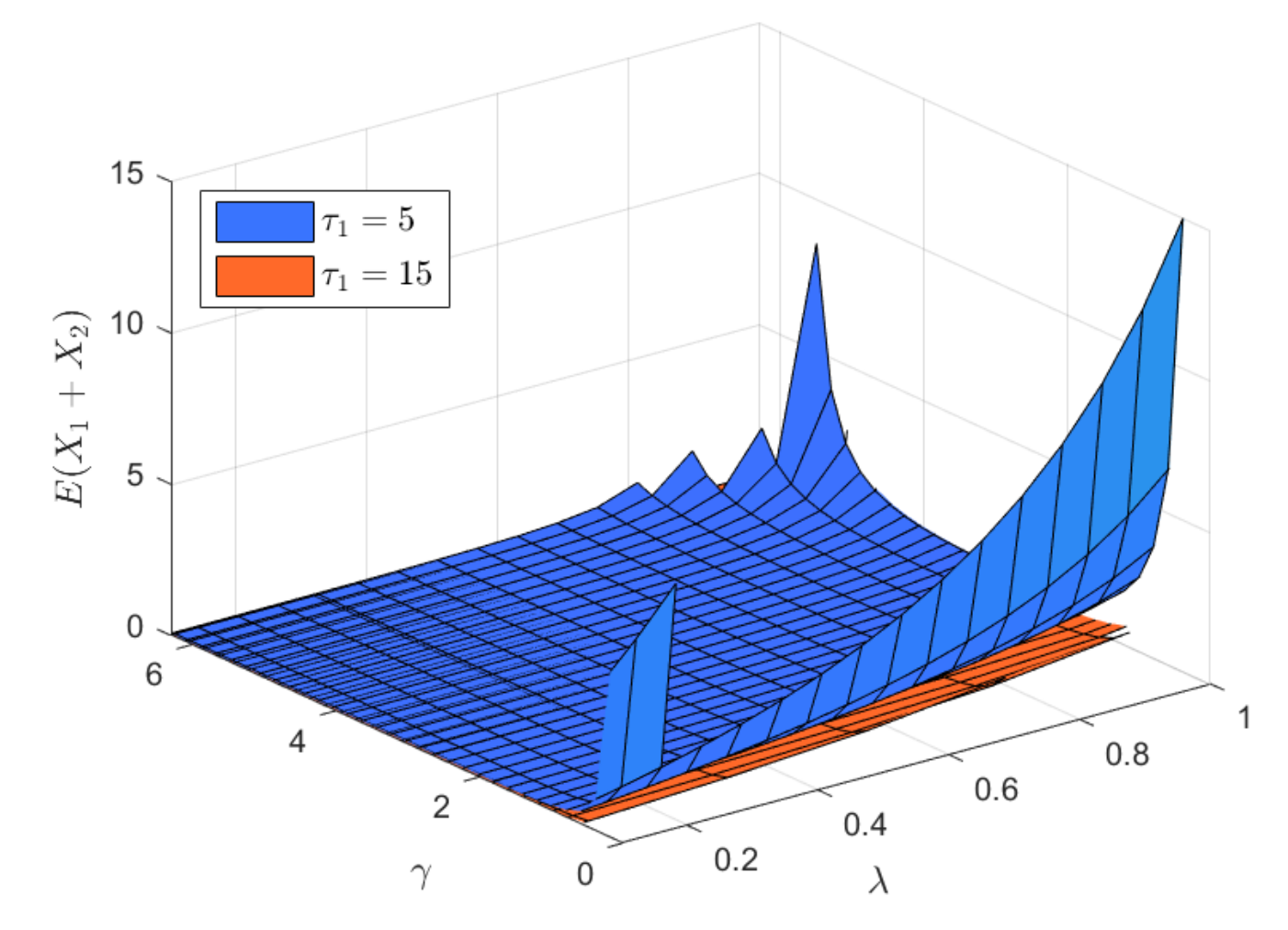}
\caption{$E(X_{1}+X_{2})$ for the symmetrical system for  $\tau_{2}=8$, $\nu=6$.}
\label{fig:10}       
\end{figure*}

\begin{figure*}
  \includegraphics[width=0.75\textwidth]{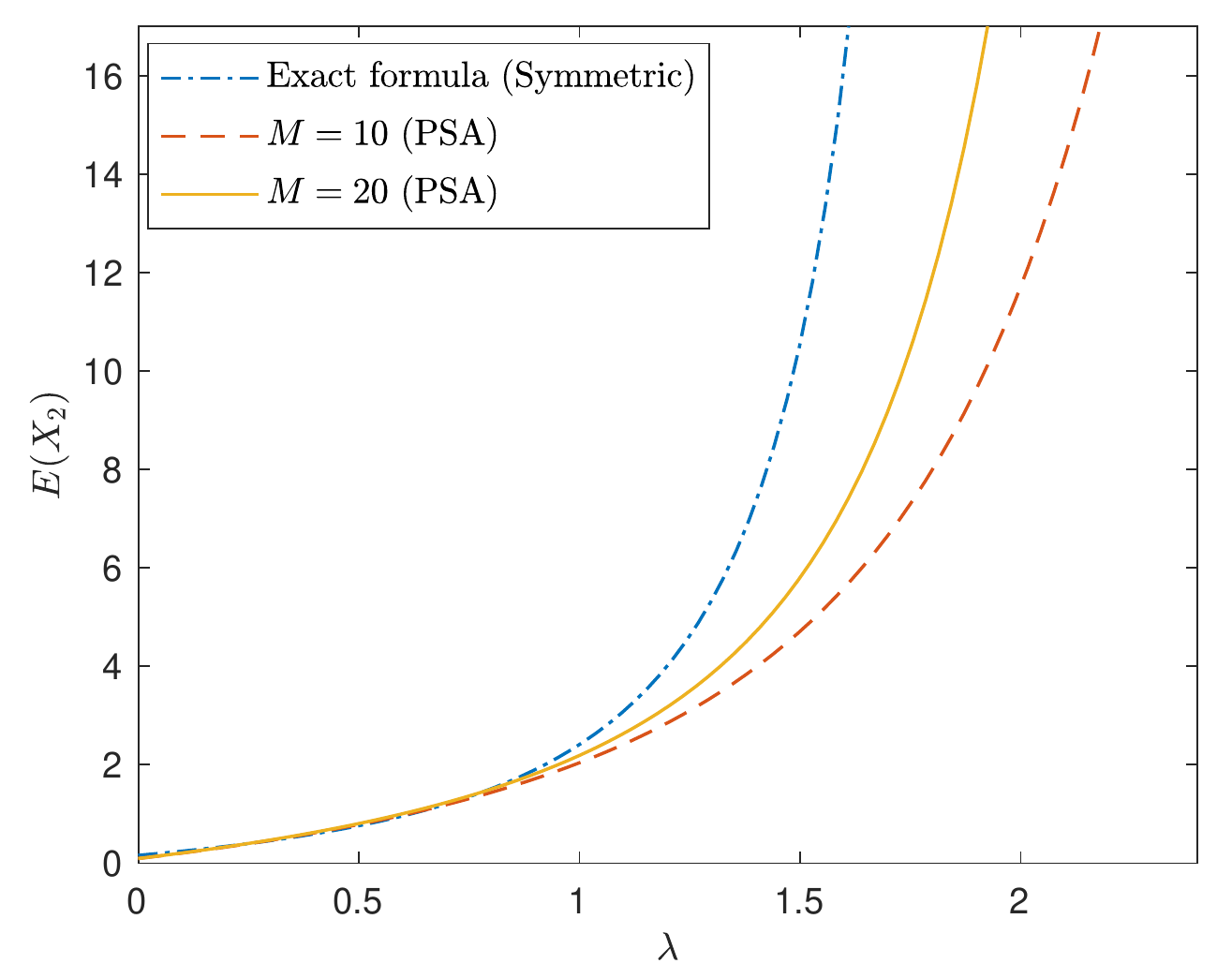}
\caption{Explicit vs PSA for $E(X_{2})$ for the symmetrical system and $(\tau_{1},\tau_{2})=(5,8)$.}
\label{fig:9}       
\end{figure*}

\section{Conclusion}\label{conc}
This work focused on the stationary analysis of a certain class of Markov-modulated reflected random walk with potential applications in the modelling of two-node queueing networks with coupled queues and service interruptions, and in priority retrial systems with coupled orbit queues. We presented three methodological approaches based on the generating function approach: $i)$ The power series approximation (PSA) method, under which we derived power series expansions of the pgfs of the stationary joint distributions for each state of the phase process, and reveal the flexibility of the PSA approach in even more complicated setting. $ii)$ The theory of Riemann boundary value problems. $iii)$ We also provided explicit expressions for the first moments of the stationary distribution when we assume the \textit{symmetrical assumption} at the phase $J(t)=0$. This result is obtained without solving a boundary value problem. 

Our developed technique is general enough to deal with the analysis of Markov-modulated coupled queueing systems. Examples are the standard, as well as the G-queueing networks with service interruptions and coupled queues, as well as priority retrial systems with two coupled orbit queues. Potential extensions of the PSA refer to the case of three (or more) queues. Unfortunately, the theory of boundary value problems
has not been developed for problems with more than two queues.

The overall conclusion is that PSA, the theory of boundary value problems as well as the corresponding analysis under the symmetry assumption seems to be adapted to any Markov-modulated two-dimensional processes for which the \textit{matrix} functional equation \eqref{matrix} can be transformed into a \textit{scalar} functional equation of the form given in \eqref{fe}.
\section*{Acknowledgements}
The author would like to thank the Editor and the anonymous Reviewers for the careful reading of the manuscript and the insightful remarks and input, which helped to improve the original exposition.

\appendix
\section{Appendix A}\label{app1}
\begin{theorem}
For every $|x|=1$, $x\neq 1$, $G(x,y)=0$ has a unique zero, say $Y(x)$, in the disc $|y|<1$.
\end{theorem}
\begin{proof}
In order to enhance the readability, we consider the case where $\theta_{i,j}=0$, $i,j=1,\ldots,N$, so that $T(x,y)$ is given in \eqref{edd} (Subsection \ref{subse}). Analogous arguments can be applied even in the general case, where $\theta_{i,j}>0$, $i,j=1,\ldots,N$.

The proof is based on the application of Rouch\'e's theorem \cite{titc,adanrouche}. 
Note that $G(x,y)=0$ is rewritten as $G_{0}(x,y)=h(x,y):=y\sum_{j=1}^{N}\frac{\theta_{j,0}\theta_{0,j}A_{0,j}(x,y)A_{j,0}(x,y)}{D_{j}(x,y)}$, where
\begin{displaymath}
\begin{array}{rl}
G_{0}(x,y)=&yD_{0}(x,y)+\sum_{i=0}^{\infty}q^{(2)}_{i,-1}(y-x^{i}).
\end{array}
\end{displaymath}

It is easy to see that $G_{0}(x,y)=f(x,y)-g(x,y)$, with
\begin{displaymath}
\begin{array}{rl}
f(x,y)=&y[\theta_{0,.}+\sum_{i=1}^{\infty}q_{i,0}(0)(1-x^{i})+\sum_{i=1}^{\infty}\sum_{j=1}^{\infty}q_{i,j}(0)+\sum_{j=1}^{\infty}q_{0,j}(0)+\sum_{i=0}^{\infty}q_{i,-1}(0)],\\
g(x,y)=&\sum_{i=1}^{\infty}\sum_{j=1}^{\infty}q_{i,j}(0)x^{i}y^{j}+\sum_{j=1}^{\infty}q_{0,j}(0)y^{j}+\sum_{i=0}^{\infty}q_{i,-1}(0)x^{i}.
\end{array}
\end{displaymath}
%
We first show that $G_{0}(x,y)=0$ has a unique root in $|y|<1$, for $|x|=1$, $x\neq1$. Then, for $|x|=1$, $x\neq1$,
\begin{displaymath}
\begin{array}{rl}
|f(x,y)|&=|y||\theta_{0,.}+\sum_{i=1}^{\infty}q_{i,0}(0)(1-x^{i})+\sum_{i=1}^{\infty}\sum_{j=1}^{\infty}q_{i,j}(0)+\sum_{j=1}^{\infty}q_{0,j}(0)+\sum_{i=0}^{\infty}q_{i,-1}(0)|\\
&\geq |y|[\theta_{0,.}+\sum_{i=1}^{\infty}\sum_{j=1}^{\infty}q_{i,j}(0)+\sum_{j=1}^{\infty}q_{0,j}(0)+\sum_{i=0}^{\infty}q_{i,-1}(0)]\\
&>|y|[\sum_{i=1}^{\infty}\sum_{j=1}^{\infty}q_{i,j}(0)+\sum_{j=1}^{\infty}q_{0,j}(0)+\sum_{i=0}^{\infty}q_{i,-1}(0)],\vspace{2mm}\\
|g(x,y)|&\leq\sum_{i=1}^{\infty}\sum_{j=1}^{\infty}q_{i,j}(0)|x|^{i}|y|^{j}+\sum_{j=1}^{\infty}q_{0,j}(0)|y|^{j}+\sum_{i=0}^{\infty}q_{i,-1}(0)|x|^{i}\\&=\sum_{i=1}^{\infty}\sum_{j=1}^{\infty}q_{i,j}(0)|y|^{j}+\sum_{j=1}^{\infty}q_{0,j}(0)|y|^{j}+\sum_{i=0}^{\infty}q_{i,-1}(0)
\end{array}
\end{displaymath}
Then, for all $y$ such that $|y|=1$ have that
\begin{displaymath}
|g(x,y)|\leq\sum_{i=1}^{\infty}\sum_{j=1}^{\infty}q_{i,j}(0)+\sum_{j=1}^{\infty}q_{0,j}(0)+\sum_{i=0}^{\infty}q_{i,-1}(0)<|f(x,y)|,\,|y|=1,|x|=1,x\neq 1,
\end{displaymath}
which implies by Rouch\'{e}'s theorem, see, e.g. \cite{titc}, that $G_{0}(x,y)$ has as many zeros, counted according to their multiplicity,
inside $|y| = 1$ as $f(x, y)$. Since $f(x, y)$ has only one zero of multiplicity
1 at $y = 0$, yields that for every $x$ with $|x| = 1$, $x\neq 1$, $G_{0}(x,y) = 0$
has one root inside $|y| = 1$.

Now, note that
\begin{displaymath}
|G_{0}(x,y)|=|f(x,y)-g(x,y)|\geq||f(x,y)|-|g(x,y)||>\theta_{0,.}.
\end{displaymath}
Moreover,
\begin{displaymath}
\begin{array}{rl}
|-h(x,y)|=&|-y\sum_{j=1}^{N}\frac{\theta_{j,0}\theta_{0,j}A_{0,j}(x,y)A_{j,0}(x,y)}{D_{j}(x,y)}|\leq|y|\sum_{j=1}^{N}\frac{\theta_{j,0}\theta_{0,j}|A_{0,j}(x,y)||A_{j,0}(x,y)|}{|D_{j}(x,y)|}\\<&\sum_{j=1}^{N}\frac{\theta_{j,0}\theta_{0,j}}{\theta_{j,0}}=\theta_{0,.},
\end{array}
\end{displaymath}
since $|A_{0,j}(x,y)|=1$, $|A_{j,0}(x,y)|=1$, $|D_{j}(x,y)|\geq\theta_{j,0}$, for $|x|=1$, $|y|=1$, $x\neq1$. Thus,
\begin{displaymath}
|-h(x,y)|<\theta_{0,.}<|G_{0}(x,y)|,\,|x|=1, |y|=1, x\neq1.
\end{displaymath}
Therefore, Rouch\'e's theorem implies that $G_{0}(x,y)=y\sum_{j=1}^{N}\frac{\theta_{j,0}\theta_{0,j}}{D_{j}(x,y)}$, i.e., $G(x,y)=0$ has the same number of zeros for every $x$ with $|x| = 1$, $x\neq 1$, inside $|y|=1$, as $G_{0}(x,y)$, which we shown that has exactly one. Thus, denote this zero as $y=Y(x)$, $|x| = 1$, $x\neq 1$, with $|Y(x)|<1$.
\end{proof}
\section{Appendix B}\label{app2}
\paragraph{Proof of Theorem \ref{th1}:} To enhance the readability, we consider the case where $\theta_{i,j}=0$, $i,j=1,\ldots,N$, and considered the QBD version of $\{Z(t);t\geq0\}$, i.e, $\mathbb{H}=\{-1,0,1\}$, $\mathbb{H}^{+}=\{0,1\}$. Analogous arguments can be applied even in the general case. The proof is based on the application of Rouch\'e's theorem \cite{titc,adanrouche}. 

Note that $U(gs,gs^{-1})=0$ is written as $R(gs,gs^{-1})=g^{2}\sum_{j=1}^{N}\frac{\theta_{j,0}\theta_{0,j}}{D_{j}(gs,gs^{-1})}$. We first show that $R(gs,gs^{-1})=0$ has a single root in $|g|\leq 1$, $|s|=1$. Indeed, $R(gs,gs^{-1}=0)$ is rewritten as
\begin{displaymath}
\begin{array}{rl}
g^{2}=L(gs,gs^{-1}):=\frac{q_{-1,0}(0)gs^{-1}+q_{0,-1}(0)gs}{D_{0}(gs,gs^{-1})+q_{-1,0}(0)+q_{0,-1}(0)+q_{-1,1}(0)(1-s^{-2})+q_{1,-1}(0)(1-s^{2})}.
\end{array}
\end{displaymath}
Note that the denominator of $L(gs,gs^{-1})$ never vanishes for $|g|\leq 1$, $|s|=1$ and in particular,
\begin{displaymath}
\begin{array}{c}
|D_{0}(gs,gs^{-1})+q_{-1,0}(0)+q_{0,-1}(0)+q_{-1,1}(0)(1-s^{-2})+q_{1,-1}(0)(1-s^{2})|\\
>\theta_{0,.}+q_{-1,0}(0)+q_{0,-1}(0).
\end{array}
\end{displaymath}
Therefore,
\begin{displaymath}
\begin{array}{rl}
|L(gs,gs^{-1})|\leq&\frac{q_{-1,0}(0)|g|+q_{0,-1}(0)|g|}{|D_{0}(gs,gs^{-1})+q_{-1,0}(0)+q_{0,-1}(0)+q_{-1,1}(0)(1-s^{-2})+q_{1,-1}(0)(1-s^{2})|}\\
<&\frac{q_{-1,0}(0)+q_{0,-1}(0)}{\theta_{0,.}+q_{-1,0}(0)+q_{0,-1}(0)}<1=|g|^{2}.
\end{array}
\end{displaymath}
Thus, by applying Rouch\'e's theorem $R(gs,gs^{-1})=0$ has a single zero in $|g|<1$, $|s|=1$. Note also that
\begin{displaymath}
\begin{array}{r}
|R(gs,gs^{-1})|\geq|g|^{2}(\theta_{0.}+q_{1,0}(1-gs)+q_{0,1}(1-gs^{-1}+q_{1,1}(0))(1-|g|^{2}))>\theta_{0,.}.
\end{array}
\end{displaymath}
On the other hand,
\begin{displaymath}
\begin{array}{l}
|g^{2}\sum_{j=1}^{N}\frac{\theta_{j,0}\theta_{0,j}A_{0,j}(gs,gs^{-1})A_{j,0}(gs,gs^{-1})}{D_{j}(gs,gs^{-1})}|\leq\sum_{j=1}^{N}\frac{\theta_{j,0}\theta_{0,j}}{|D_{j}(gs,gs^{-1})|}<\theta_{0,.}<|R(gs,gs^{-1})|.
\end{array}
\end{displaymath}
Therefore, since $R(gs,gs^{-1})$ has a single zero in $|g|<1$, $|s=1|$, Rouch\'e's theorem states that $R(gs,gs^{-1})=g^{2}\sum_{j=1}^{N}\frac{\theta_{j,0}\theta_{0,j}A_{0,j}(gs,gs^{-1})A_{j,0}(gs,gs^{-1})}{D_{j}(gs,gs^{-1})}$ has a single root in $|g|<1$, $|s=1|$. Equivalently, $U(gs,gs^{-1})=0$ has a a single root in $|g|<1$, $|s=1|$.

\bibliographystyle{spmpsci}      


%
%

\end{document}